%% file: 00_main_BB.tex
\documentclass[11pt]{article}

\usepackage[a4paper,margin=1.15in]{geometry}
\usepackage{amsmath,amssymb,amsthm,mathtools}
\usepackage[colorlinks=true,citecolor=blue,linkcolor=blue,urlcolor=blue]{hyperref}

\hypersetup{
  pdftitle={Logarithmic oscillatory multipliers and log-subdyadic square functions},
  pdfauthor={Vicente Vergara},
  pdfsubject={Fourier multipliers and logarithmic square functions},
  pdfkeywords={Fourier multipliers, square functions, logarithmic oscillation, log-subdyadic decompositions, weighted inequalities}
}

\theoremstyle{plain}
\newtheorem{theorem}{Theorem}[section]
\newtheorem{proposition}[theorem]{Proposition}
\newtheorem{lemma}[theorem]{Lemma}
\newtheorem{corollary}[theorem]{Corollary}

\theoremstyle{definition}
\newtheorem{definition}[theorem]{Definition}
\newtheorem{hypothesis}[theorem]{Hypothesis}

\theoremstyle{remark}
\newtheorem{remark}[theorem]{Remark}

\title{Logarithmic oscillatory multipliers and log-subdyadic square functions}
\author{Vicente Vergara\footnote{Department of Mathematics, Faculty of Physical and Mathematical Sciences, University of Concepci\'on, Concepci\'on, Chile. \texttt{vvergaraa@udec.cl}}}
\date{} 

\providecommand{\subjclass}[2][]{%
  \par\smallskip\noindent\textit{#1 Mathematics Subject Classification.} #2\par
}
\providecommand{\keywords}[1]{%
  \par\smallskip\noindent\textit{Keywords.} #1\par
}

\begin{document}

\maketitle

\begin{abstract}
We develop square-function estimates for Fourier multipliers whose local
oscillation scale is
\[
        \rho(R)=\frac{R}{(\log R)^{\gamma-1}},\qquad \gamma>1.
\]
This scale lies strictly between the dyadic scale and every fixed
power-subdyadic scale at high frequency. For high-frequency symbols satisfying a localized Sobolev
condition on balls of radius comparable to $\rho(R)$, we prove a
pointwise square-function estimate and a weighted $L^2$ multiplier
inequality. The associated logarithmic geometric maximal operator has
an exact strong-type threshold, including its critical equality.
After adjoining a smooth compactly supported low-frequency part,
the weighted theorem yields strong $L^p$ bounds at the critical
logarithmic decay. For the model
\[
        L(\xi)=\frac12\log(e^2+|\xi|^2),\qquad
        m_{\gamma,\beta}(\xi)=L(\xi)^{-\beta}e^{iL(\xi)^\gamma},
\]
we obtain boundedness on $L^p$, $1<p<\infty$, if and only if
\[
        \beta\geq d(\gamma-1)\left|\frac12-\frac1p\right|.
\]
The relation with classical critical Besov multiplier criteria is also
discussed.
\end{abstract}




\input{10_intro_BB}
\input{20_prelim_BB}
\input{30_decomposition_BB}
\input{40_pointwise_BB}
\input{50_weighted_BB}
\input{60_examples_BB}
\input{70_lp_bounds_BB}

\bibliographystyle{plain}
\bibliography{BB_log}

\end{document}

%% file: 10_intro_BB.tex
\section{Introduction}
\label{sec:introduction}

The study of oscillatory and strongly singular Fourier multipliers includes
foundational contributions of Hirschman and Wainger, followed by the work of
H\"ormander, Fefferman, Fefferman--Stein, Sj\"olin, Miyachi, and others
\cite{Hirschman1959MultiplierTransformations,
Wainger1965SpecialTrigonometricSeries,
Hormander1960TranslationInvariant,
Fefferman1970StronglySingular,
FeffermanStein1972HpSpaces,
Sjolin1981OscillatingKernels,
Miyachi1980FourierMultipliersHp,
Miyachi1981Singular}. A useful way to organize the different oscillatory
regimes is through the variation of a radial phase $\Phi$ across a dyadic
annulus. At frequency $R$, this variation is measured, at the level of the
first derivative, by
\[
        \Omega_\Phi(R):=R|\Phi'(R)|,
\]
while the frequency scale on which the phase changes by an amount comparable
to one is
\[
        \rho_\Phi(R)
        :=
        |\Phi'(R)|^{-1}
        =
        \frac{R}{\Omega_\Phi(R)}.
\]
When $\Omega_\Phi(R)$ remains bounded, the dyadic resolution is compatible
with the first-order variation of the phase. Once
$\Omega_\Phi(R)\to\infty$, the oscillation is no longer resolved at the
dyadic scale and a finer frequency decomposition becomes natural.

For the power phase $\Phi(R)=R^\alpha$, with $\alpha>0$, one has
\[
        \Omega_\Phi(R)\sim R^\alpha,
        \qquad
        \rho_\Phi(R)\sim R^{1-\alpha}.
\]
Beltran and Bennett developed a subdyadic square-function theory adapted to
this fixed power geometry
\cite{BeltranBennett2017Subdyadic}. Their framework combines localized
Sobolev conditions on balls of radius comparable to $R^{1-\alpha}$ with
pointwise square-function estimates and weighted inequalities governed by
geometrically defined maximal operators. This weighted perspective is
related to earlier geometric control of singular Fourier multipliers
\cite{Bennett2014OptimalControl} and to maximal operators associated with
tangential approach regions
\cite{NagelStein1984MaximalApproachRegions,GarciaSoria1999ApproachRegions}.
A more recent approach to weighted multiplier estimates is the multiscale
sparse domination theory of Beltran, Roos and Seeger
\cite[Section 7.2]{BeltranRoosSeeger2024MultiScale}, which also treats
subdyadic multiplier classes.
A complementary development of the power-subdyadic geometry from a
time-frequency perspective is given in \cite{VergaraSubdyadicGabor}, where
frequency-adaptive Gabor frames and modulation spaces are used to study
H\"ormander--Miyachi multipliers.

The regime considered here lies strictly between the dyadic setting and
every fixed power-subdyadic geometry. Fix
\[
        \Phi(R)=(\log R)^\gamma,
        \qquad
        \gamma>1,
\]
for $R$ sufficiently large. Then
\[
        \Omega_\Phi(R)
        =
        R|\Phi'(R)|
        \sim
        (\log R)^{\gamma-1},
\]
so the variation of the phase across a dyadic annulus is unbounded. On the
other hand,
\[
        (\log R)^{\gamma-1}
        =
        o(R^\alpha)
        \qquad
        \text{for every fixed }\alpha>0.
\]
Thus the phase produces genuinely subdyadic oscillation, but its complexity
grows more slowly than every positive power of the frequency. Equivalently,
its natural frequency scale
\begin{equation}
\label{eq:rho-log}
        \rho(R)
        =
        \frac{R}{(\log R)^{\gamma-1}}
\end{equation}
satisfies
\[
        R^{1-\alpha}
        \ll
        \rho(R)
        \ll
        R
        \qquad
        \text{for every fixed }\alpha>0.
\]
Consequently, this logarithmic regime is neither resolved by the classical
dyadic geometry nor represented by any fixed member of the power-subdyadic
family.

The purpose of the present paper is to develop the pointwise and weighted
square-function architecture of Beltran--Bennett in this slowly divergent
regime. The spatial scale dual to \eqref{eq:rho-log} is
\begin{equation}
\label{eq:a-gamma}
        a_\gamma(t)
        =
        \rho(t^{-1})^{-1}
        =
        t(\log(1/t))^{\gamma-1}.
\end{equation}
On the dyadic annulus $|\xi|\sim2^k$, the frequency radius is
\[
        \rho(2^k)
        \sim
        \frac{2^k}{k^{\gamma-1}},
\]
and an adapted partition contains on the order of
$k^{d(\gamma-1)}$ pieces. The partition scale, the physical aperture, and
the localized Sobolev normalization therefore vary together with the
dyadic generation.

The guiding model is
\[
        L(\xi)
        =
        \frac12\log(e^2+|\xi|^2)
\]
and
\begin{equation}
\label{eq:model-multiplier}
        m_{\gamma,\beta}(\xi)
        =
        L(\xi)^{-\beta}e^{iL(\xi)^\gamma}.
\end{equation}
The regularization $L$ is smooth on $\mathbb{R}^d$ and satisfies
$L(\xi)=\log|\xi|+O(|\xi|^{-2})$ as $|\xi|\to\infty$. At high frequency,
the derivatives of the oscillatory factor are governed by
$\rho(|\xi|)^{-1}$ rather than by $|\xi|^{-1}$. The localized logarithmic
Miyachi condition introduced below records this derivative scale on balls
of radius comparable to $\rho(R)$.

The square-function method belongs to the Littlewood--Paley framework of
Stein \cite{Stein1970SingularIntegrals,Stein1970Topics}. It is also related
to square functions for nonstandard frequency decompositions, including the
Rubio de Francia inequality for arbitrary intervals
\cite{RubioDeFrancia1985LittlewoodPaley} and square functions associated with
radial multipliers \cite{LeeRogersSeeger2014SquareFunctions}. Here the
frequency geometry is radial and scale-dependent, and the associated physical
approach region has aperture $a_\gamma(t)$.

Let $\phi$ be the fixed Littlewood--Paley function from Section
\ref{sec:preliminaries}. The logarithmic square function is
\[
        g_{\log,\gamma,\beta}(f)(x)
        =
        \left(
        \int_0^{t_0}
        \int_{|y-x|\leq a_\gamma(t)}
        |f*\phi_t(y)|^2
        \frac{\mathrm{d}y}{a_\gamma(t)^d}
        (\log(1/t))^{2\beta}
        \frac{\mathrm{d}t}{t}
        \right)^{1/2}.
\]
Its robust counterpart is
\[
        g^*_{\log,\gamma,\beta,\lambda}(f)(x)
        =
        \left(
        \int_0^{t_0}
        \int_{\mathbb{R}^d}
        |f*\phi_t(y)|^2
        K_t^\lambda(x-y)
        (\log(1/t))^{2\beta}
        \frac{\mathrm{d}t}{t}
        \right)^{1/2},
\]
where
\[
        K_t^\lambda(z)
        =
        a_\gamma(t)^{-d}
        \left(1+\frac{|z|}{a_\gamma(t)}\right)^{-d\lambda}.
\]
The factor $(\log(1/t))^{2\beta}$ compensates for logarithmic symbol decay
at frequency $|\xi|\sim t^{-1}$. The precise definitions are given in
Definition \ref{def:log-square-functions}.

The first main result is a pointwise square-function estimate. If $m$
satisfies the localized logarithmic Miyachi condition with Sobolev order
$\sigma>d/2$ and $\lambda=2\sigma/d$, then
\begin{equation}
\label{eq:intro-main-estimate}
        g_{\log,\gamma,\beta}(T_mf)(x)
        \leq
        Cg^*_{\log,\gamma,0,\lambda}(f)(x).
\end{equation}
The constant depends on $d,\gamma,\beta,\sigma,\lambda$, the localized
Miyachi constant, and the fixed cutoffs, but not on $f$, $x$, the scale, or
the particular adapted ball. The full statement is Theorem
\ref{thm:log-bb-pointwise}.

The proof separates the logarithmic geometry from the local multiplier
estimate. On each dyadic annulus we construct a lattice partition at radius
$\rho(2^k)$, prove decoupling and recoupling inequalities for that partition,
and combine them with a one-ball Sobolev estimate. The local stability of
$\rho$ is used to keep all constants uniform as the number of pieces grows
with $k$. These ingredients are contained in Propositions
\ref{prop:log-decoupling}, \ref{prop:log-recoupling}, and
\ref{prop:local-multiplier-estimate}. The argument therefore requires
simultaneous control of the partition scale, the physical aperture, and the
localized Sobolev normalization.

The second main result is a weighted $L^2$ estimate for a high-frequency
multiplier $T_{m_{\mathrm{hi}}}$. The relevant geometric maximal operator is
\[
        M_{\log,\gamma,\beta}w(x)
        =
        \sup_{\substack{0<t<t_0\\ |y-x|\leq a_\gamma(t)}}
        (\log(1/t))^{-2\beta}A_tw(y),
\]
where $A_tw=\eta_t*w$ is a smooth average at scale $t$. The inverse
logarithmic factor matches the weight in
$g_{\log,\gamma,\beta}$. If $\beta\geq0$,
$\operatorname{supp}m_{\mathrm{hi}}\subset\{|\xi|\geq2t_0^{-1}\}$ and
$m_{\mathrm{hi}}$ satisfies the localized hypothesis with $\sigma>d/2$, then
\begin{equation}
\label{eq:intro-weighted-bound}
        \int_{\mathbb{R}^d}
        |T_{m_{\mathrm{hi}}}f(x)|^2w(x)\,\mathrm{d}x
        \lesssim
        \int_{\mathbb{R}^d}
        |f(x)|^2
        M^2\bigl(M_{\log,\gamma,\beta}M^4w\bigr)(x)
        \,\mathrm{d}x.
\end{equation}
This follows from the pointwise estimate, a reverse logarithmic
square-function inequality, and the forward weighted estimate for the robust
square function. The argument uses weighted Littlewood--Paley estimates in
the form developed by Wilson \cite{Wilson1989WeightedSquareFunction}. The
full statement is Theorem \ref{thm:weighted-log-multiplier}.

The logarithmic maximal operator is bounded on $L^r$ precisely when
\[
        2\beta\geq\frac{d(\gamma-1)}{r},\qquad 1<r<\infty.
\]
The critical strong bound follows from a covering argument that uses
monotonicity of the logarithmic aperture and cancels the volume of each
enlarged ball against the critical normalization. Below the threshold,
a test on characteristic functions of balls shows that even the weak-type
estimate fails (Lemma \ref{lem:log-maximal-Lr} and Proposition
\ref{prop:log-maximal-necessity}).

Combining this sharp maximal bound with the weighted theorem gives the
strong $L^p$ estimate, including the critical equality. More precisely,
let $1<p<\infty$ and $m=m_{\mathrm{lo}}+m_{\mathrm{hi}}$, where
$m_{\mathrm{lo}}$ is smooth and compactly supported and
$m_{\mathrm{hi}}$ satisfies the localized logarithmic Miyachi condition
with $\sigma>d/2$ and is supported in $\{|\xi|\geq2t_0^{-1}\}$. Then
\[
        \|T_mf\|_{L^p}\lesssim\|f\|_{L^p}
        \qquad\text{if}\qquad
        \beta\geq d(\gamma-1)\left|\frac12-\frac1p\right|.
\]
For the model \eqref{eq:model-multiplier}, the condition is also
necessary, by a localized oscillatory test with a uniformly nondegenerate
Hessian. Thus the sufficient threshold for the localized class cannot
be lowered in general (Corollary \ref{cor:model-lp}).

Critical Besov conditions in the H\"ormander multiplier theory were
studied by Seeger \cite{Seeger1988LimitHormander,Seeger1989NearL1};
see also the discussion of the corresponding Lorentz-space estimates in
\cite[Appendix A]{SeegerTrebels2019LorentzSobolev}. In particular, the
unweighted multiplier bound above can also be obtained from Seeger's
strong-type criterion \cite[Corollary 1.2]{Seeger1990SingularConvolution}.
Remark \ref{rem:seeger-comparison} gives the reduction from our localized
Sobolev condition and discusses the relation with sparse bounds away from
the critical decay. The principal contributions of the present approach
are the pointwise square-function inequality and the explicit geometric
weighted control in \eqref{eq:intro-weighted-bound}. The sharp unweighted
bounds follow from this control and the maximal theorem.

The paper is organized as follows. Section \ref{sec:preliminaries} introduces
the logarithmic geometry, the square functions, and the localized Miyachi
condition. Section \ref{sec:log-decomposition} constructs the adapted lattice
partition and proves the decoupling and recoupling estimates. Section
\ref{sec:pointwise-estimate} establishes the local multiplier estimate and
the pointwise theorem. Section \ref{sec:weighted} proves the weighted
inequalities. Section \ref{sec:model-multiplier} verifies the model symbol
and proves its sharp decay threshold. Section \ref{sec:lp-bounds} proves
the sharp maximal bounds and their Lebesgue space consequences, and relates
the localized hypothesis to Besov multiplier criteria and sparse bounds.

%% file: 20_prelim_BB.tex
\section{Preliminaries and logarithmic geometry}
\label{sec:preliminaries}

Throughout the paper $d\geq1$, $\gamma>1$, and $R_0\geq e^2$ is chosen
sufficiently large, depending only on the structural constants under
discussion. Constants denoted by $C$ may change from line to line.
Non-universal constants have their dependencies stated explicitly.

\subsection{The logarithmic scale}
\label{subsec:log-scale}

For $R\geq R_0$, define
\[
 \rho(R)=\frac{R}{(\log R)^{\gamma-1}}.
\]
The scale $\rho(R)$ is comparable to $|\Phi'(R)|^{-1}$ for the phase
$\Phi(R)=(\log R)^\gamma$. We also define
\[
 a_\gamma(t)=\rho(t^{-1})^{-1}=t(\log(1/t))^{\gamma-1},
 \qquad 0<t<e^{-1}.
\]
The square functions below are truncated to $0<t<t_0$, where $t_0$ is fixed
in Subsection \ref{subsec:square-functions} after the Littlewood--Paley
cutoff has been chosen.

\begin{lemma}[Local stability of the logarithmic scale]
\label{lem:rho-local-stability}
Let $A>0$. If $R_0$ is sufficiently large depending on $A$ and $\gamma$, then
\[
 |R'-R|\leq A\rho(R), \qquad R\geq R_0,
\]
implies
\[
 C_A^{-1}\rho(R)\leq \rho(R')\leq C_A\rho(R).
\]
Moreover, the ratio $\rho(R')/\rho(R)$ tends to $1$ uniformly under this constraint as $R\to\infty$.
\end{lemma}

\begin{proof}
Since
\[
 \rho'(R)=(\log R)^{1-\gamma}\left(1-\frac{\gamma-1}{\log R}\right),
\]
we have $|\rho'(R)|\leq C\rho(R)/R$ for $R\geq R_0$. The assumption gives
\[
 \frac{|R'-R|}{R}\leq A(\log R)^{1-\gamma},
\]
which tends to zero. Hence $R'\sim R$ and $\log R'\sim \log R$, with constants depending only on $A$ and $\gamma$. Substitution in the definition of $\rho$ gives the claim.
\end{proof}

\subsection{Log-subdyadic balls}
\label{subsec:log-balls}

\begin{definition}[Log-subdyadic ball]
\label{def:log-subdyadic-ball}
Let $B\subset \mathbb{R}^d$ be a Euclidean ball. Write $r(B)$ for its radius and
\[
 R_B=\operatorname{dist}(B,0).
\]
The ball $B$ is called log-subdyadic of order $\gamma$ if $R_B\geq R_0$ and
\[
 c_0\rho(R_B)\leq r(B)\leq C_0\rho(R_B),
\]
where $0<c_0<C_0<\infty$ are fixed structural constants.
\end{definition}

\begin{definition}[Adapted normalized bump]
\label{def:adapted-bump}
Let $N$ be a fixed integer, chosen sufficiently large compared with the Sobolev exponents used below. A normalized bump adapted to a ball $B$ is a function $\Psi_B\in C_c^\infty(\mathbb{R}^d)$ such that
\[
 \operatorname{supp}\Psi_B\subset C_1 B
\]
for a fixed dilation constant $C_1$, and
\[
 |D^\nu\Psi_B(\xi)|\leq C_\nu r(B)^{-|\nu|},
 \qquad |\nu|\leq N.
\]
All constants in statements involving adapted bumps are required to be uniform over this class.
\end{definition}

\begin{lemma}[Stability inside adapted balls]
\label{lem:ball-scale-stability}
If $B$ is log-subdyadic and $\xi,\eta\in 2B$, then
\[
 |\xi|\sim |\eta|\sim R_B,
 \qquad
 \rho(|\xi|)\sim \rho(|\eta|)\sim \rho(R_B)\sim r(B).
\]
The implicit constants depend only on $d$, $\gamma$, and on the structural constants in Definition \ref{def:log-subdyadic-ball}.
\end{lemma}

\begin{proof}
For $\xi\in 2B$, the distance from $|\xi|$ to $R_B$ is bounded by a fixed multiple of $r(B)$, hence by a fixed multiple of $\rho(R_B)$. Lemma \ref{lem:rho-local-stability} gives $\rho(|\xi|)\sim \rho(R_B)$, and the remaining comparabilities follow in the same way.
\end{proof}

\subsection{Fourier multipliers and maximal operators}

For a measurable symbol $m$ of at most polynomial growth, we write
\[
        T_m f=\mathcal{F}^{-1}(m\widehat f)=:(m\widehat f)^\vee
\]
initially on Schwartz functions. This convention includes the logarithmically
growing model symbols that occur when $\beta<0$. We write $M$ for the
Hardy--Littlewood
maximal operator,
\[
        Mf(x)
        =
        \sup_{r>0}
        \frac{1}{|B(x,r)|}
        \int_{B(x,r)} |f(y)|\,\mathrm{d}y,
\]
and $M^j$ for its $j$-fold composition.

\subsection{Square functions}
\label{subsec:square-functions}

Fix a radial function $\phi\in\mathcal{S}(\mathbb{R}^d)$ such that
\[
 \operatorname{supp}\widehat{\phi}
 \subset \{\xi:1/2\leq |\xi|\leq 2\},
 \qquad
 \int_0^\infty \widehat{\phi}(t\xi)\,\frac{\mathrm{d}t}{t}=1,
 \quad \xi\neq0,
\]
and set $\phi_t(x)=t^{-d}\phi(x/t)$. After fixing $k_0$ with
$2^{k_0}\geq R_0$, choose an integer $k_1\geq k_0+2$ such that
\[
        t_0=2^{-k_1}<e^{-1}.
\]
Then
\[
        \operatorname{supp}\widehat{\phi_t}
        \subset \{\xi:|\xi|\geq 2^{k_0+1}\},
        \qquad 0<t<t_0,
\]
so the square functions do not detect the low-frequency portion of a
multiplier.

Define the truncated Calder\'on reproducing multiplier
\[
        p_{\mathrm{hi}}(\xi)
        =
        \int_0^{t_0}\widehat{\phi}(t\xi)\,\frac{\mathrm{d}t}{t},
        \qquad
        P_{\mathrm{hi}}=T_{p_{\mathrm{hi}}}.
\]
The support and normalization of $\widehat\phi$ imply
\[
        p_{\mathrm{hi}}(\xi)=0
        \quad\text{if }|\xi|\leq (2t_0)^{-1},
        \qquad
        p_{\mathrm{hi}}(\xi)=1
        \quad\text{if }|\xi|\geq 2t_0^{-1}.
\]
We also fix a radial cutoff $\chi_{\mathrm{hi}}\in C^\infty(\mathbb{R}^d)$
such that
\[
        \chi_{\mathrm{hi}}(\xi)=0
        \quad\text{for }|\xi|\leq 2t_0^{-1},
        \qquad
        \chi_{\mathrm{hi}}(\xi)=1
        \quad\text{for }|\xi|\geq 4t_0^{-1}.
\]
In particular,
\[
        p_{\mathrm{hi}}\chi_{\mathrm{hi}}
        =
        \chi_{\mathrm{hi}}.
\]

\begin{definition}[Log-subdyadic square functions]
\label{def:log-square-functions}
For $\beta\in\mathbb{R}$, define
\[
 g_{\log,\gamma,\beta}(f)(x)
 =\left(
 \int_0^{t_0}\int_{|y-x|\leq a_\gamma(t)}
 |f*\phi_t(y)|^2
 \frac{\mathrm{d}y}{a_\gamma(t)^d}
 (\log(1/t))^{2\beta}\frac{\mathrm{d}t}{t}
 \right)^{1/2}.
\]
For $\lambda>1$, define
\[
 g^{*}_{\log,\gamma,\beta,\lambda}(f)(x)
 =\left(
 \int_0^{t_0}\int_{\mathbb{R}^d}
 |f*\phi_t(y)|^2
 \left(1+\frac{|x-y|}{a_\gamma(t)}\right)^{-d\lambda}
 \frac{\mathrm{d}y}{a_\gamma(t)^d}
 (\log(1/t))^{2\beta}\frac{\mathrm{d}t}{t}
 \right)^{1/2}.
\]
\end{definition}

Define the normalized kernel
\begin{equation}
\label{eq:robust-kernel}
 K_t^\lambda(z)=a_\gamma(t)^{-d}
 \left(1+\frac{|z|}{a_\gamma(t)}\right)^{-d\lambda}.
\end{equation}
For $\lambda>1$, $\|K_t^\lambda\|_{L^1}\leq C$. We shall also use the
elementary convolution stability
\begin{equation}
\label{eq:kernel-convolution}
 K_t^\lambda*K_t^\lambda(z)\leq C K_t^\lambda(z),
\end{equation}
uniformly in $0<t<t_0$. Indeed, after the change of variables
$u=a_\gamma(t)^{-1}y$, this reduces to
\[
        \int_{\mathbb{R}^d}
        (1+|u|)^{-d\lambda}
        (1+|v-u|)^{-d\lambda}\,\mathrm{d}u
        \lesssim
        (1+|v|)^{-d\lambda}.
\]
To prove this, cover the domain of integration by the two regions
$|u|\geq |v|/2$ and $|v-u|\geq |v|/2$. On the first region,
\[
        (1+|u|)^{-d\lambda}
        \lesssim
        (1+|v|)^{-d\lambda},
\]
and the remaining factor is integrable; the second region is identical. The
condition $\lambda>1$ is exactly what gives integrability in dimension
$d$.

\subsection{Logarithmic Miyachi conditions}
\label{subsec:log-miyachi}

\begin{hypothesis}[Localized logarithmic Miyachi condition]
\label{hyp:localized-log-miyachi}
Let $\sigma>d/2$. A multiplier $m$ satisfies the localized logarithmic Miyachi condition of order $(\gamma,\beta,\sigma)$ if, for every log-subdyadic ball $B$, every adapted normalized bump $\Psi_B$, and every $0\leq \theta\leq \sigma$,
\[
 (\log R_B)^\beta\rho(R_B)^\theta |B|^{-1/2}
 \|m\Psi_B\|_{\dot H^\theta}\leq C_m.
\]
The constant $C_m$ is independent of $B$ and $\Psi_B$.
\end{hypothesis}

\begin{lemma}[Stability under complex conjugation]
\label{lem:log-miyachi-conjugation}
If $m$ satisfies Hypothesis \ref{hyp:localized-log-miyachi}, then
$\overline m$ satisfies the same hypothesis with the same constant $C_m$.
\end{lemma}

\begin{proof}
Let $B$ be log-subdyadic and let $\Psi_B$ be an adapted normalized bump.
Then $\overline{\Psi_B}$ is an adapted normalized bump with the same support
and derivative bounds, and
\[
        \overline m\,\Psi_B
        =
        \overline{m\,\overline{\Psi_B}}.
\]
For every $\theta\geq0$, complex conjugation preserves the homogeneous
Sobolev norm:
\[
        \|\overline F\|_{\dot H^\theta}
        =
        \|F\|_{\dot H^\theta}.
\]
Indeed,
$\widehat{\overline F}(\xi)=\overline{\widehat F(-\xi)}$, and the multiplier
$|\xi|^\theta$ is even. Applying Hypothesis
\ref{hyp:localized-log-miyachi} to the adapted bump
$\overline{\Psi_B}$ proves the claim.
\end{proof}

\begin{remark}[Comparison with classical symbol classes]
The relation between the logarithmic Miyachi condition and the classical
Mikhlin class with logarithmic decay is discussed in Remark
\ref{rem:comparison-symbol-classes}, where the model multiplier is also shown
to make the inclusion strict.
\end{remark}

\begin{lemma}[Pointwise derivative condition implies localized Sobolev control]
\label{lem:pointwise-to-sobolev}
Assume that, for every multi-index $\nu$ with $|\nu|\leq N$, where $N>\sigma$,
\[
 |D^\nu m(\xi)|\leq C_\nu (\log |\xi|)^{-\beta}\rho(|\xi|)^{-|\nu|},
 \qquad |\xi|\geq R_0.
\]
Then $m$ satisfies Hypothesis \ref{hyp:localized-log-miyachi}.
\end{lemma}

\begin{proof}
For integer $\theta=n$, Leibniz' rule and Lemma \ref{lem:ball-scale-stability} give, for $|\nu|=n$,
\[
 \|D^\nu(m\Psi_B)\|_{L^2}
 \leq C_\nu (\log R_B)^{-\beta}\rho(R_B)^{-n}|B|^{1/2}.
\]
This yields the required estimate for integer $n$. Fractional orders follow by interpolation between adjacent integer Sobolev estimates. The constants depend on finitely many derivative bounds for $m$ and on the bump normalization.
\end{proof}

%% file: 30_decomposition_BB.tex
\section{Log-subdyadic decompositions}
\label{sec:log-decomposition}

\subsection{Lattice partitions}
\label{subsec:lattice-partitions}

Let $A_k=\{\xi:2^k\leq |\xi|\leq 2^{k+1}\}$, and set
\[
 r_k=\rho(2^k)\sim \frac{2^k}{k^{\gamma-1}}
\]
for $k\geq k_0$, where $2^{k_0}\geq R_0$. Choose a smooth dyadic partition $\{\eta_k\}_{k\geq k_0}$ on high frequencies with $\operatorname{supp}\eta_k\subset\{\xi:|\xi|\sim 2^k\}$. Choose $\nu\in C_c^\infty(\mathbb{R}^d)$ such that
\[
 \sum_{\ell\in\mathbb{Z}^d}\nu(\zeta-\ell)=1,
 \qquad \zeta\in\mathbb{R}^d.
\]
Set
\[
 \nu_{k,\ell}(\xi)=\nu(r_k^{-1}\xi-\ell),
 \qquad
 \psi_{k,\ell}(\xi)=\eta_k(\xi)\nu_{k,\ell}(\xi).
\]
Then
\[
 \sum_{k\geq k_0}\sum_{\ell\in\mathbb{Z}^d}\psi_{k,\ell}(\xi)=1,
 \qquad |\xi|\geq 2^{k_0+1},
\]
and each $\psi_{k,\ell}$ is supported in a ball of radius comparable to $r_k$. The number of relevant balls in $A_k$ is $O(k^{d(\gamma-1)})$.

Throughout this section $\psi_{k,\ell}$, $\eta_k$, and $\nu_{k,\ell}$ denote frequency cutoffs, and $\check\psi_{k,\ell}$, $\check\eta_k$, and $\check\nu_{k,\ell}$ denote their inverse Fourier transforms. Thus $f*\check\psi_{k,\ell}$ is the Fourier projection associated with the cutoff $\psi_{k,\ell}$.

\begin{remark}[Why the lattice is retained]
\label{rem:lattice-needed}
The decoupling estimate only uses bounded overlap and scale stability. The recoupling estimate below uses the lattice structure through a local Bessel inequality, so it is stated for this particular partition.
\end{remark}

\subsection{Decoupling}
\label{subsec:decoupling}

\begin{proposition}[Log-subdyadic decoupling]
\label{prop:log-decoupling}
Let $\lambda>1$. If $f_{k,\ell}$ has Fourier support contained in $\operatorname{supp}\psi_{k,\ell}$, then
\[
 g_{\log,\gamma,\beta}\left(\sum_{k,\ell} f_{k,\ell}\right)(x)^2
 \leq C\sum_{k,\ell}g^{*}_{\log,\gamma,\beta,\lambda}(f_{k,\ell})(x)^2.
\]
The constant depends on $d$, $\gamma$, $\beta$, $\lambda$, and on the fixed cutoffs.
\end{proposition}

\begin{proof}
It is enough to prove the estimate for Schwartz functions and finite sums of
frequency pieces; the general case follows by the usual density and limiting
arguments. We first choose a globally nonnegative band-limited majorant. Let
$\vartheta\in C_c^\infty(\mathbb{R}^d)$ satisfy
$\check\vartheta(0)\neq0$. After a fixed dilation, the function
\[
        \Phi(z)=|\check\vartheta(\varepsilon z)|^2
\]
with $\varepsilon>0$ sufficiently small satisfies
\[
        \Phi\geq0\quad\text{on }\mathbb{R}^d,
        \qquad
        \Phi\geq c>0\quad\text{on }|z|\leq1,
        \qquad
        \operatorname{supp}\widehat\Phi
        \subset\{|\zeta|\leq C_0\}.
\]
Thus the corresponding smoothed square function majorizes
$g_{\log,\gamma,\beta}$.

Fix $t$, and write
\[
        h_{k,\ell}=f_{k,\ell}*\phi_t,
        \qquad
        W_{t,x}(y)
        =
        a_\gamma(t)^{-d}
        \Phi\left(\frac{x-y}{a_\gamma(t)}\right).
\]
The Fourier transform of $W_{t,x}$ in the $y$-variable is supported in a
ball of radius
$O(a_\gamma(t)^{-1})=O(\rho(t^{-1}))$. Moreover,
$\widehat{\phi_t}$ restricts the frequency variable to
$|\xi|\sim t^{-1}$, so only indices $(k,\ell)$ with
$2^k\sim t^{-1}$ contribute.

For each such $(k,\ell)$, let $\mathcal N(k,\ell)$ denote the set of indices
$(k',\ell')$ for which the cross term
\[
        \int_{\mathbb{R}^d}
        h_{k,\ell}(y)\overline{h_{k',\ell'}(y)}W_{t,x}(y)
        \,\mathrm{d}y
\]
can be nonzero. The support properties above, the scale relation
$r_k\sim\rho(t^{-1})$, and the bounded overlap of the lattice partition give
\[
        \#\mathcal N(k,\ell)\leq C
\]
uniformly in $k,\ell,t$. Since $W_{t,x}\geq0$, Cauchy's inequality gives
\[
\begin{aligned}
&\int_{\mathbb{R}^d}
\left|\sum_{k,\ell}h_{k,\ell}(y)\right|^2
W_{t,x}(y)\,\mathrm{d}y \\
&\qquad\lesssim
\sum_{k,\ell}
\int_{\mathbb{R}^d}|h_{k,\ell}(y)|^2W_{t,x}(y)\,\mathrm{d}y .
\end{aligned}
\]
Finally, $\Phi$ is Schwartz, so for the fixed $\lambda>1$,
\[
        W_{t,x}(y)
        \lesssim
        K_t^\lambda(x-y).
\]
Multiplying by $(\log(1/t))^{2\beta}$ and integrating in $t$ proves the
asserted estimate.
\end{proof}

\subsection{Recoupling}
\label{subsec:recoupling}

\begin{lemma}[Local lattice Bessel inequality]
\label{lem:lattice-bessel}
For the lattice partition above,
\[
 \sum_{\ell\in\mathbb{Z}^d}|F*\check\nu_{k,\ell}(y)|^2
 \leq C |F|^2*\Theta_k(y),
\]
where $\Theta_k$ is a nonnegative Schwartz majorant at spatial scale
$r_k^{-1}$ with uniformly bounded $L^1$ norm. The constant $C$ is independent
of $k$, $F$, and $y$.
\end{lemma}

\begin{proof}
By dilation it suffices to prove the estimate for $r_k=1$. In that case write $\nu_\ell(\xi)=\nu(\xi-\ell)$. For fixed $y$, the quantity $F*\check\nu_\ell(y)$ can be written, up to harmless normalizing constants, as
\[
 e^{2\pi i\ell\cdot y}\widehat{H_y}(\ell),
 \qquad
 H_y(u)=F(u)\check\nu(y-u).
\]
Parseval's identity on the torus, followed by the Poisson summation formula, gives
\[
 \sum_{\ell\in\mathbb{Z}^d}|F*\check\nu_\ell(y)|^2
 \leq C\int_{[0,1]^d}\sum_{m\in\mathbb{Z}^d}
 |F(u+m)|^2|\check\nu(y-u-m)|\,
 \sum_{m'\in\mathbb{Z}^d}|\check\nu(y-u-m')|\,\mathrm{d}u.
\]
Since $\check\nu$ is Schwartz,
\[
 \sup_{v\in\mathbb{R}^d}\sum_{m\in\mathbb{Z}^d}|\check\nu(v-m)|<\infty.
\]
Therefore
\[
 \sum_{\ell\in\mathbb{Z}^d}|F*\check\nu_\ell(y)|^2
 \leq C\int_{\mathbb{R}^d}|F(u)|^2|\check\nu(y-u)|\,\mathrm{d}u.
\]
Rescaling gives the stated estimate with
\[
        \Theta_k(x)=r_k^d\Theta(r_kx),
\]
where $\Theta\geq0$ is a fixed Schwartz majorant of $|\check\nu|$.
\end{proof}

\begin{proposition}[Log-subdyadic recoupling]
\label{prop:log-recoupling}
Let $\lambda>1$. For the lattice partition $\{\psi_{k,\ell}\}$,
\[
 \sum_{k,\ell}g^{*}_{\log,\gamma,\beta,\lambda}(f*\check\psi_{k,\ell})(x)^2
 \leq C g^{*}_{\log,\gamma,\beta,\lambda}(f)(x)^2.
\]
The constant depends only on $d$, $\gamma$, $\lambda$, and on the fixed
cutoffs.
\end{proposition}

\begin{proof}
As before, it suffices to argue first for Schwartz functions and finite partial sums. By definition and Fubini's theorem, the left-hand side is
\[
 \int_0^{t_0}\int_{\mathbb{R}^d}
 \sum_{k,\ell}|f*\phi_t*\check\eta_k*\check\nu_{k,\ell}(y)|^2
 K_t^\lambda(x-y)(\log(1/t))^{2\beta}\,\mathrm{d}y\frac{\mathrm{d}t}{t}.
\]
Only $k$ with $2^k\sim t^{-1}$ contribute. Applying Lemma \ref{lem:lattice-bessel} with
$F=f*\phi_t*\check\eta_k$ gives
\[
        \sum_\ell
        |f*\phi_t*\check\eta_k*\check\nu_{k,\ell}|^2
        \leq
        C |f*\phi_t*\check\eta_k|^2*\Theta_k.
\]
Since $\|\check\eta_k\|_{L^1}\lesssim1$, Cauchy's inequality gives
\[
        |f*\phi_t*\check\eta_k|^2
        \lesssim
        |f*\phi_t|^2*|\check\eta_k|.
\]
Consequently,
\[
        \sum_\ell
        |f*\phi_t*\check\eta_k*\check\nu_{k,\ell}|^2
        \leq
        C |f*\phi_t|^2*|\check\eta_k|*\Theta_k.
\]
The kernels $\Theta_k$ are nonnegative Schwartz majorants at spatial scale
$r_k^{-1}\sim a_\gamma(t)$, while $|\check\eta_k|$ is at spatial scale
$2^{-k}\sim t\leq a_\gamma(t)$. Since only $O(1)$ indices $k$ satisfy
$2^k\sim t^{-1}$, and since $K_t^\lambda$ is stable under convolution
with uniformly $L^1$-normalized Schwartz kernels at scales bounded by
$C a_\gamma(t)$, we have
\[
        \sum_{2^k\sim t^{-1}}
        |\check\eta_k|*\Theta_k*K_t^\lambda
        \leq
        C K_t^\lambda.
\]
Here the constant depends on $d,\gamma,\lambda$ and on the fixed cutoffs.
Substitution gives the desired estimate.
\end{proof}

%% file: 40_pointwise_BB.tex
\section{The pointwise square-function estimate}
\label{sec:pointwise-estimate}

\subsection{The local multiplier estimate}
\label{subsec:local-multiplier}

\begin{proposition}[Local logarithmic multiplier estimate]
\label{prop:local-multiplier-estimate}
Let $\sigma>d/2$, set $\lambda=2\sigma/d$, and assume that $m$
satisfies Hypothesis \ref{hyp:localized-log-miyachi}. Let $B$ be a
log-subdyadic ball, and let $\psi_B$ be a normalized frequency cutoff
adapted to $B$, with adaptation constants uniform in $B$. Then
\[
 g^{*}_{\log,\gamma,\beta,\lambda}\bigl(T_m(f*\check\psi_B)\bigr)(x)
 \leq C g^{*}_{\log,\gamma,0,\lambda}(f*\check\psi_B)(x).
\]
The constant depends on $d$, $\gamma$, $\beta$, $\sigma$, the
localized Miyachi constant of $m$, and the fixed cutoff constants, but is
uniform in $B$, $f$, and $x$.
\end{proposition}

\begin{proof}
Choose a normalized bump $\varphi_B$ adapted to $B$, equal to one on
$\operatorname{supp}\psi_B$, supported in a fixed dilation of $B$, and
satisfying
\[
 |D^\nu\varphi_B(\xi)|\leq C_\nu r(B)^{-|\nu|}
\]
with constants uniform in $B$. Since $\varphi_B\equiv1$ on
$\operatorname{supp}\psi_B$, the localized multiplier relevant to
$f*\check\psi_B$ is
\[
 K_B=\mathcal{F}^{-1}(m\varphi_B),
\]
and
\[
        T_m(f*\check\psi_B)=K_B*(f*\check\psi_B).
\]

Hypothesis \ref{hyp:localized-log-miyachi} with $\theta=0$, together with Plancherel, gives
\[
 \|K_B\|_{L^2}\leq C(\log R_B)^{-\beta}r(B)^{d/2}.
\]
Hypothesis \ref{hyp:localized-log-miyachi} with $\theta=\sigma$ gives the corresponding weighted kernel estimate
\[
 \||z|^\sigma K_B(z)\|_{L^2_z}
 \leq C(\log R_B)^{-\beta}r(B)^{d/2-\sigma}.
\]
Here we use the Plancherel correspondence between Sobolev regularity in
frequency and spatial moments of the inverse Fourier transform:
\[
        \||z|^\sigma K_B\|_{L^2_z}
        \lesssim
        \|m\varphi_B\|_{\dot H^\sigma}.
\]
For integer $\sigma$ this follows by differentiating $m\varphi_B$ in
frequency. Fractional $\sigma$ follows from the usual fractional Sobolev
form of Plancherel.

Fix $t$ for which $(f*\check\psi_B)*\phi_t$ is nonzero. By the support of $\widehat{\phi_t}$ and the support of $\psi_B$, we have
\[
 t^{-1}\sim R_B,
 \qquad
 r(B)\sim \rho(t^{-1})=a_\gamma(t)^{-1},
 \qquad
 \log R_B\sim \log(1/t).
\]
Set
\[
 F_t=(f*\check\psi_B)*\phi_t.
\]
Split
\[
 K_B*F_t(y)=I_{\mathrm{near}}(y)+I_{\mathrm{far}}(y),
\]
where
\[
 I_{\mathrm{near}}(y)=\int_{|z|\leq a_\gamma(t)}K_B(z)F_t(y-z)\,\mathrm{d}z
\]
and
\[
 I_{\mathrm{far}}(y)=\int_{|z|>a_\gamma(t)}K_B(z)F_t(y-z)\,\mathrm{d}z.
\]
For the near term, Cauchy's inequality and the $L^2$ bound for $K_B$ give
\[
 |I_{\mathrm{near}}(y)|^2
 \leq C(\log R_B)^{-2\beta}r(B)^d
 \int_{|z|\leq a_\gamma(t)}|F_t(y-z)|^2\,\mathrm{d}z.
\]
Since $r(B)^d\sim a_\gamma(t)^{-d}$, and $K_t^\lambda(z)\sim a_\gamma(t)^{-d}$ for $|z|\leq a_\gamma(t)$, this yields
\[
 |I_{\mathrm{near}}(y)|^2
 \leq C(\log R_B)^{-2\beta}
 \int_{\mathbb{R}^d}|F_t(y-z)|^2K_t^\lambda(z)\,\mathrm{d}z.
\]
For the far term, Cauchy's inequality with weight $|z|^\sigma$ gives
\[
 |I_{\mathrm{far}}(y)|^2
 \leq \||z|^\sigma K_B\|_{L^2}^2
 \int_{|z|>a_\gamma(t)}|z|^{-2\sigma}|F_t(y-z)|^2\,\mathrm{d}z.
\]
Using the weighted kernel estimate and $r(B)\sim a_\gamma(t)^{-1}$, the right-hand side is bounded by
\[
 C(\log R_B)^{-2\beta}
 \int_{|z|>a_\gamma(t)}
 |F_t(y-z)|^2 a_\gamma(t)^{-d}
 \left(\frac{|z|}{a_\gamma(t)}\right)^{-2\sigma}
 \mathrm{d}z.
\]
Since $2\sigma=d\lambda$, this is bounded by
\[
 C(\log R_B)^{-2\beta}
 \int_{\mathbb{R}^d}|F_t(y-z)|^2K_t^\lambda(z)\,\mathrm{d}z.
\]
Combining the near and far estimates, and multiplying by the output weight $(\log(1/t))^{2\beta}$, gives
\[
 (\log(1/t))^{2\beta}|K_B*F_t(y)|^2
 \leq C\int_{\mathbb{R}^d}|F_t(y-z)|^2K_t^\lambda(z)\,\mathrm{d}z.
\]
Integrating against $K_t^\lambda(x-y)\,\mathrm{d}y$, using \eqref{eq:kernel-convolution}, and then integrating in $t$, gives the result.
\end{proof}

\subsection{Assembly of the pointwise estimate}
\label{subsec:pointwise-assembly}

\begin{theorem}[Logarithmic Beltran--Bennett pointwise estimate]
\label{thm:log-bb-pointwise}
Let $\sigma>d/2$, $\lambda=2\sigma/d$, and assume that $m$ satisfies Hypothesis \ref{hyp:localized-log-miyachi}. Then
\[
 g_{\log,\gamma,\beta}(T_m f)(x)
 \leq C g^{*}_{\log,\gamma,0,\lambda}(f)(x).
\]
The constant depends on $d$, $\gamma$, $\beta$, $\sigma$, $\lambda$, the localized Miyachi constant of $m$, and the fixed cutoffs, but not on $f$ or $x$. Since the square functions only involve $0<t<t_0$, the estimate depends
only on the high-frequency portion of $m$ detected by
$\widehat{\phi_t}$.
\end{theorem}

\begin{proof}
Apply Proposition \ref{prop:log-decoupling} to $T_m f$, using the logarithmic partition of unity on the high-frequency region detected by $\phi_t$. This gives
\[
 g_{\log,\gamma,\beta}(T_m f)(x)^2
 \leq C\sum_B g^{*}_{\log,\gamma,\beta,\lambda}\bigl(T_m(f*\check\psi_B)\bigr)(x)^2.
\]
Proposition \ref{prop:local-multiplier-estimate} bounds the right-hand side by
\[
 C\sum_B g^{*}_{\log,\gamma,0,\lambda}(f*\check\psi_B)(x)^2.
\]
Finally Proposition \ref{prop:log-recoupling}, with $\beta=0$, gives
\[
 \sum_B g^{*}_{\log,\gamma,0,\lambda}(f*\check\psi_B)(x)^2
 \leq C g^{*}_{\log,\gamma,0,\lambda}(f)(x)^2.
\]
Taking square roots completes the proof.
\end{proof}

%% file: 50_weighted_BB.tex
\section{Weighted estimates}
\label{sec:weighted}

We now record the weighted estimates needed to pass from the pointwise
square-function estimate of Theorem \ref{thm:log-bb-pointwise} to weighted
$L^2$ bounds. The forward estimate for the robust square function is
standard. The reverse estimate is the logarithmic analogue of the
Beltran--Bennett reverse square-function inequality
\cite{BeltranBennett2017Subdyadic}.

Choose a radial function $\kappa\in\mathcal{S}(\mathbb{R}^d)$ such that
$\widehat\kappa=1$ on a neighborhood of
$\operatorname{supp}\widehat\phi$. Throughout this section we fix a
nonnegative radial function $\eta\in \mathcal{S}(\mathbb{R}^d)$ such that
\[
        \int_{\mathbb{R}^d}\eta(x)\,\mathrm{d}x=1,
        \qquad
        \eta(x)\geq c\,\mathbf{1}_{B(0,1)}(x),
        \qquad
        |\kappa(x)|\leq C\eta(x).
\]
Such a positive radial Schwartz majorant can be fixed once and for all. We set
\[
        \eta_t(x)=t^{-d}\eta(x/t).
\]
For a locally integrable function $w$, define the smooth averaging
operator
\[
        A_t w(y)
        =
        \eta_t*w(y).
\]

\begin{definition}[Logarithmic geometric maximal operator]
\label{def:log-maximal-candidate}
Let
\[
        a_\gamma(t)=t(\log(1/t))^{\gamma-1}.
\]
For a nonnegative locally integrable function $w$, define
\[
        M_{\log,\gamma,\beta}w(x)
        =
        \sup_{\substack{0<t<t_0\\ |y-x|\leq a_\gamma(t)}}
        (\log(1/t))^{-2\beta}A_t w(y).
\]
The operator is local in scale, consistently with the high-frequency nature
of the logarithmic square functions. The displacement of the center is
measured at the log-subdyadic aperture $a_\gamma(t)$, while the averaging
takes place at the underlying Littlewood--Paley scale $t$.
\end{definition}

\begin{remark}[Normalization]
\label{rem:maximal-normalization}
The logarithmic square function $g_{\log,\gamma,\beta}$ contains the
factor
\[
        (\log(1/t))^{2\beta}.
\]
The corresponding weighted maximal operator contains the inverse factor
\[
        (\log(1/t))^{-2\beta}.
\]
This is the logarithmic counterpart of the Beltran--Bennett power
normalization: $g_{\alpha,\beta}$ carries the factor $t^{-2\beta}$,
whereas the associated maximal operator carries the compensating factor
$t^{2\beta}$.
\end{remark}

\begin{remark}[Smooth and hard averages]
\label{rem:smooth-hard-averages}
Since $\eta\geq c\,\mathbf{1}_{B(0,1)}$, the smooth average $A_tw(y)$
dominates the normalized average over $B(y,t)$, up to a dimensional
constant. Thus the smooth maximal operator in Definition
\ref{def:log-maximal-candidate} is stronger than the corresponding
hard-ball version. The smooth formulation is convenient because it is
stable under the Littlewood--Paley smoothing step used below.
\end{remark}

\begin{proposition}[Forward weighted estimate for $g^*_{\log,\gamma,0,\lambda}$]
\label{prop:forward-log-g-star}
Let $\gamma>1$ and $\lambda>1$. Then, for every nonnegative locally
integrable weight $w$,
\[
        \int_{\mathbb{R}^d}
        g^*_{\log,\gamma,0,\lambda}(f)(x)^2w(x)\,\mathrm{d}x
        \lesssim
        \int_{\mathbb{R}^d}
        |f(x)|^2M^2w(x)\,\mathrm{d}x .
\]
The implicit constant depends only on $d,\gamma,\lambda$, and on the
fixed Littlewood--Paley cutoff.
\end{proposition}

\begin{proof}
By the definition of $g^*_{\log,\gamma,0,\lambda}$ and Fubini's theorem,
\[
\begin{aligned}
        \int_{\mathbb{R}^d}
        g^*_{\log,\gamma,0,\lambda}(f)(x)^2w(x)\,\mathrm{d}x
        &=
        \int_0^{t_0}
        \int_{\mathbb{R}^d}
        |f*\phi_t(y)|^2
        (K_t^\lambda*w)(y)
        \,\mathrm{d}y\,\frac{\mathrm{d}t}{t}.
\end{aligned}
\]
Since $\lambda>1$, the kernels $K_t^\lambda$ are uniformly integrable
and are dominated, after dyadic annular decomposition, by averages over
balls centered at $y$. Hence
\[
        K_t^\lambda*w(y)\lesssim Mw(y)
\]
uniformly in $0<t<t_0$. Therefore
\[
        \int_{\mathbb{R}^d}
        g^*_{\log,\gamma,0,\lambda}(f)(x)^2w(x)\,\mathrm{d}x
        \lesssim
        \int_0^{t_0}
        \int_{\mathbb{R}^d}
        |f*\phi_t(y)|^2Mw(y)
        \,\mathrm{d}y\,\frac{\mathrm{d}t}{t}.
\]
The classical weighted Littlewood--Paley estimate gives
\[
        \int_{\mathbb{R}^d}
        \int_0^\infty
        |f*\phi_t(y)|^2
        \frac{\mathrm{d}t}{t}
        W(y)\,\mathrm{d}y
        \lesssim
        \int_{\mathbb{R}^d}
        |f(y)|^2MW(y)\,\mathrm{d}y
\]
for every nonnegative locally integrable weight $W$; see
\cite{Wilson1989WeightedSquareFunction} and also
\cite[Proposition 15]{BeltranBennett2017Subdyadic}. Applying this with
$W=Mw$, and
discarding the harmless restriction $0<t<t_0$, yields
\[
        \int_{\mathbb{R}^d}
        g^*_{\log,\gamma,0,\lambda}(f)(x)^2w(x)\,\mathrm{d}x
        \lesssim
        \int_{\mathbb{R}^d}
        |f(y)|^2M^2w(y)\,\mathrm{d}y .
\]
\end{proof}

\begin{lemma}[Littlewood--Paley smoothing of the weight]
\label{lem:lp-scale-averaging}
Let $\phi\in\mathcal{S}(\mathbb{R}^d)$ be fixed with
\[
        \operatorname{supp}\widehat{\phi}
        \subset
        \{\xi:1/2\leq |\xi|\leq 2\}.
\]
Then, for every nonnegative locally integrable weight $w$, every
Schwartz function $f$, and every $t>0$,
\[
        \int_{\mathbb{R}^d}
        |f*\phi_t(y)|^2M^3w(y)\,\mathrm{d}y
        \lesssim
        \int_{\mathbb{R}^d}
        |f*\phi_t(y)|^2A_tM^4w(y)\,\mathrm{d}y .
\]
The implicit constant depends only on $d$ and on the fixed functions
$\phi$, $\kappa$, and $\eta$.
\end{lemma}

\begin{proof}
Let
\[
        F_t=f*\phi_t.
\]
Since $\widehat\kappa=1$ on a neighborhood of
$\operatorname{supp}\widehat\phi$, we have
\[
        F_t=F_t*\kappa_t.
\]
Cauchy's inequality and $|\kappa|\leq C\eta$ therefore give
\[
\begin{aligned}
        |F_t(y)|^2
        &\leq
        \|\kappa\|_{L^1}
        \int_{\mathbb{R}^d}
        |F_t(y-z)|^2|\kappa_t(z)|\,\mathrm{d}z \\
        &\lesssim
        (|F_t|^2*\eta_t)(y).
\end{aligned}
\]
Hence
\[
\begin{aligned}
        \int_{\mathbb{R}^d}|F_t(y)|^2M^3w(y)\,\mathrm{d}y
        &\lesssim
        \int_{\mathbb{R}^d}
        (|F_t|^2*\eta_t)(y)M^3w(y)\,\mathrm{d}y  \\
        &=
        \int_{\mathbb{R}^d}
        |F_t(z)|^2(\eta_t*M^3w)(z)\,\mathrm{d}z .
\end{aligned}
\]
Since the Hardy--Littlewood maximal operator dominates its input almost
everywhere,
\[
        M^3w \leq M^4w
        \qquad\text{a.e.}
\]
Therefore
\[
        \eta_t*M^3w(z)
        \leq
        \eta_t*M^4w(z)
        =
        A_tM^4w(z).
\]
Combining the preceding inequalities gives the claim.
\end{proof}

\begin{lemma}[Geometric averaging over logarithmic approach regions]
\label{lem:log-geometric-averaging}
For every nonnegative measurable function $F$, every nonnegative locally
integrable weight $W$, and every $0<t<t_0$,
\[
        \int_{\mathbb{R}^d}F(y)A_tW(y)\,\mathrm{d}y
        \lesssim
        \int_{\mathbb{R}^d}
        \int_{|y-x|\leq a_\gamma(t)}
        F(y)\frac{\mathrm{d}y}{a_\gamma(t)^d}
        \sup_{|z-x|\leq a_\gamma(t)}A_tW(z)
        \,\mathrm{d}x .
\]
The implicit constant depends only on $d$.
\end{lemma}

\begin{proof}
For each fixed $y$,
\[
        A_tW(y)
        \leq
        \sup_{|z-x|\leq a_\gamma(t)}A_tW(z)
\]
whenever $|x-y|\leq a_\gamma(t)$. Hence
\[
        F(y)A_tW(y)
        \lesssim
        F(y)
        \frac{1}{a_\gamma(t)^d}
        \int_{|x-y|\leq a_\gamma(t)}
        \sup_{|z-x|\leq a_\gamma(t)}A_tW(z)
        \,\mathrm{d}x .
\]
Integrating in $y$ and applying Fubini gives the claim.
\end{proof}

\begin{proposition}[Reverse weighted estimate for $g_{\log,\gamma,\beta}$]
\label{prop:reverse-log-g}
Let $\gamma>1$, let $\beta\geq0$, and let $P_{\mathrm{hi}}$ be the
truncated Calder\'on reproducing operator defined in Subsection
\ref{subsec:square-functions}. Then, for every nonnegative locally
integrable weight $w$,
\[
        \int_{\mathbb{R}^d}|P_{\mathrm{hi}}f(x)|^2w(x)\,\mathrm{d}x
        \lesssim
        \int_{\mathbb{R}^d}
        g_{\log,\gamma,\beta}(f)(x)^2
        M_{\log,\gamma,\beta}M^4w(x)\,\mathrm{d}x .
\]
The implicit constant depends only on $d$, $\gamma$, $\beta$, $t_0$, and
on the fixed functions $\phi$, $\kappa$, and $\eta$.
\end{proposition}

\begin{proof}
The identity
\[
        P_{\mathrm{hi}}f
        =
        \int_0^{t_0}f*\phi_t\,\frac{\mathrm{d}t}{t}
\]
is the truncated Calder\'on reproducing formula. Since
$t_0=2^{-k_1}$ is a dyadic endpoint, write
\[
        P_{\mathrm{hi}}f
        =
        \sum_{j=0}^{5}f_j,
        \qquad
        f_j
        =
        \sum_{\substack{k\leq-k_1-1\\ k\equiv j\!\!\!\pmod 6}}
        \int_{2^k}^{2^{k+1}}f*\phi_t\,\frac{\mathrm{d}t}{t}.
\]
The multiplier of the $k$th integral is supported in
\[
        \{\xi:2^{-k-2}\leq|\xi|\leq2^{-k+1}\}.
\]
Within each residue class modulo six these annuli are disjoint after a fixed
enlargement. Hence the randomized Calder\'on--Zygmund operators in the proof
of \cite[Proposition 15]{BeltranBennett2017Subdyadic} retain kernel bounds
uniform in the restricted index set $k\leq-k_1-1$. No boundary term appears,
because $t_0$ is an endpoint of the dyadic decomposition. The same
randomization, weighted Calder\'on--Zygmund estimate, and Khintchine argument
therefore give
\[
        \int_{\mathbb{R}^d}|P_{\mathrm{hi}}f(x)|^2w(x)\,\mathrm{d}x
        \lesssim
        \int_0^{t_0}
        \int_{\mathbb{R}^d}
        |f*\phi_t(y)|^2M^3w(y)
        \,\mathrm{d}y\,\frac{\mathrm{d}t}{t}.
\]
Applying Lemma \ref{lem:lp-scale-averaging} at each scale $t$, we obtain
\[
        \int_{\mathbb{R}^d}|P_{\mathrm{hi}}f(x)|^2w(x)\,\mathrm{d}x
        \lesssim
        \int_0^{t_0}
        \int_{\mathbb{R}^d}
        |f*\phi_t(y)|^2A_tM^4w(y)
        \,\mathrm{d}y\,\frac{\mathrm{d}t}{t}.
\]
We now apply Lemma \ref{lem:log-geometric-averaging} with
\[
        F(y)=|f*\phi_t(y)|^2,
        \qquad
        W=M^4w.
\]
This gives
\[
\begin{aligned}
        \int_{\mathbb{R}^d}|P_{\mathrm{hi}}f(x)|^2w(x)\,\mathrm{d}x
        \lesssim
        \int_0^{t_0}
        \int_{\mathbb{R}^d}
        \int_{|y-x|\leq a_\gamma(t)}
        |f*\phi_t(y)|^2
        \frac{\mathrm{d}y}{a_\gamma(t)^d}
        \sup_{|z-x|\leq a_\gamma(t)}A_tM^4w(z)
        \,\mathrm{d}x\,\frac{\mathrm{d}t}{t}.
\end{aligned}
\]
Insert
\[
        1
        =
        (\log(1/t))^{2\beta}(\log(1/t))^{-2\beta}.
\]
By Definition \ref{def:log-maximal-candidate},
\[
        \sup_{|z-x|\leq a_\gamma(t)}
        (\log(1/t))^{-2\beta}A_tM^4w(z)
        \leq
        M_{\log,\gamma,\beta}M^4w(x).
\]
Therefore
\[
\begin{aligned}
        \int_{\mathbb{R}^d}|P_{\mathrm{hi}}f(x)|^2w(x)\,\mathrm{d}x
        &\lesssim
        \int_{\mathbb{R}^d}
        \bigg(
        \int_0^{t_0}
        \int_{|y-x|\leq a_\gamma(t)}
        |f*\phi_t(y)|^2
        \frac{\mathrm{d}y}{a_\gamma(t)^d}
        (\log(1/t))^{2\beta}
        \frac{\mathrm{d}t}{t}
        \bigg)  \\
        &\hspace{3.5cm}\times
        M_{\log,\gamma,\beta}M^4w(x)
        \,\mathrm{d}x .
\end{aligned}
\]
The expression in parentheses is exactly
$g_{\log,\gamma,\beta}(f)(x)^2$, and the proof follows.
\end{proof}

\begin{theorem}[Weighted logarithmic multiplier estimate]
\label{thm:weighted-log-multiplier}
Let $\gamma>1$, let $\beta\geq0$, and let $\sigma>d/2$. Set
\[
        \lambda=\frac{2\sigma}{d}.
\]
Let $m_{\mathrm{hi}}$ be supported in
\[
        \{\xi\in\mathbb{R}^d:|\xi|\geq 2t_0^{-1}\}
\]
and satisfy the localized logarithmic Miyachi condition of Hypothesis
\ref{hyp:localized-log-miyachi}. Then, for every nonnegative locally
integrable weight $w$,
\[
        \int_{\mathbb{R}^d}|T_{m_{\mathrm{hi}}}f(x)|^2w(x)\,\mathrm{d}x
        \lesssim
        \int_{\mathbb{R}^d}|f(x)|^2
        M^2\bigl(M_{\log,\gamma,\beta}M^4w\bigr)(x)\,\mathrm{d}x .
\]
The implicit constant depends on $d$, $\gamma$, $\beta$, $\sigma$,
$t_0$, the localized Miyachi constant $C_m$, and on the fixed functions and
cutoffs $\phi$, $\kappa$, and $\eta$.
\end{theorem}

\begin{proof}
By the support assumption on $m_{\mathrm{hi}}$ and the identity
$p_{\mathrm{hi}}=1$ on $\{|\xi|\geq2t_0^{-1}\}$, we have
\[
        P_{\mathrm{hi}}T_{m_{\mathrm{hi}}}f=T_{m_{\mathrm{hi}}}f.
\]
Applying Proposition \ref{prop:reverse-log-g} to $T_{m_{\mathrm{hi}}}f$
therefore gives
\[
        \int_{\mathbb{R}^d}|T_{m_{\mathrm{hi}}}f(x)|^2w(x)\,\mathrm{d}x
        \lesssim
        \int_{\mathbb{R}^d}
        g_{\log,\gamma,\beta}(T_{m_{\mathrm{hi}}}f)(x)^2
        M_{\log,\gamma,\beta}M^4w(x)
        \,\mathrm{d}x .
\]
By the pointwise estimate of Theorem \ref{thm:log-bb-pointwise},
\[
        g_{\log,\gamma,\beta}(T_{m_{\mathrm{hi}}}f)(x)
        \lesssim
        g^*_{\log,\gamma,0,\lambda}(f)(x).
\]
Therefore
\[
        \int_{\mathbb{R}^d}|T_{m_{\mathrm{hi}}}f(x)|^2w(x)\,\mathrm{d}x
        \lesssim
        \int_{\mathbb{R}^d}
        g^*_{\log,\gamma,0,\lambda}(f)(x)^2
        M_{\log,\gamma,\beta}M^4w(x)
        \,\mathrm{d}x .
\]
Finally apply Proposition \ref{prop:forward-log-g-star} with the weight
\[
        W=M_{\log,\gamma,\beta}M^4w.
\]
This gives
\[
        \int_{\mathbb{R}^d}|T_{m_{\mathrm{hi}}}f(x)|^2w(x)\,\mathrm{d}x
        \lesssim
        \int_{\mathbb{R}^d}
        |f(x)|^2M^2W(x)
        \,\mathrm{d}x,
\]
which is the claimed estimate.
\end{proof}

\begin{remark}[Low frequencies]
\label{rem:weighted-low-frequency}
Theorem \ref{thm:weighted-log-multiplier} is stated for the high-frequency
piece governed by the logarithmic square functions. If
$m=m_{\mathrm{lo}}+m_{\mathrm{hi}}$, where $m_{\mathrm{lo}}$ is a smooth
compactly supported multiplier and $m_{\mathrm{hi}}$ is supported in
$\{|\xi|\geq2t_0^{-1}\}$ and satisfies Hypothesis
\ref{hyp:localized-log-miyachi}, then the low-frequency term is controlled
by the standard weighted theory for smooth compactly supported multipliers.
Thus the theorem gives the high-frequency part of the global inhomogeneous
weighted estimate, with the low-frequency contribution supplied by these
standard bounds.
\end{remark}

%% file: 60_examples_BB.tex
\section{The logarithmic model multiplier}
\label{sec:model-multiplier}

We verify that the logarithmic model multiplier belongs to the localized
logarithmic Miyachi class introduced above, and record the resulting
Lebesgue space consequences.

Let
\[
        L(\xi)=\frac12\log(e^2+|\xi|^2),
        \qquad
        m_{\gamma,\beta}(\xi)
        =
        L(\xi)^{-\beta}e^{iL(\xi)^\gamma},
        \qquad
        \gamma>1.
\]
The function $L$ is smooth on $\mathbb{R}^d$, satisfies $L\geq1$, and obeys
\[
        L(\xi)\sim \log|\xi|
        \qquad\text{for }|\xi|\gg1.
\]

\begin{proposition}[The model satisfies the logarithmic Miyachi condition]
\label{prop:model-log-miyachi}
Let $\gamma>1$, $\beta\in\mathbb{R}$, and let
\[
        L(\xi)=\frac12\log(e^2+|\xi|^2),
        \qquad
        m_{\gamma,\beta}(\xi)
        =
        L(\xi)^{-\beta}e^{iL(\xi)^\gamma}.
\]
Then, in the high-frequency region, $m_{\gamma,\beta}$ satisfies the
pointwise logarithmic Miyachi estimates
\[
        |D^\nu m_{\gamma,\beta}(\xi)|
        \lesssim_{\nu,\gamma,\beta}
        (\log|\xi|)^{-\beta}
        \left(
        \frac{(\log|\xi|)^{\gamma-1}}{|\xi|}
        \right)^{|\nu|}
\]
for every multi-index $\nu$. Consequently,
$m_{\gamma,\beta}$ satisfies Hypothesis
\ref{hyp:localized-log-miyachi} for every finite Sobolev order
$\sigma$.
\end{proposition}

\begin{proof}
Write, for $R\geq1$,
\[
        L(R)=\frac12\log(e^2+R^2),
        \qquad
        a(R)=L(R)^{-\beta},
        \qquad
        \Phi(R)=L(R)^\gamma.
\]
Then
\[
        m_{\gamma,\beta}(\xi)
        =
        a(|\xi|)e^{i\Phi(|\xi|)}
\]
and this radial function is smooth at the origin because $L(\xi)$ is a
smooth function of $|\xi|^2$.

For every integer $j\geq1$,
\[
        |L^{(j)}(R)|
        \lesssim_j
        R^{-j},
        \qquad R\gg1,
\]
and therefore
\[
        |\Phi^{(j)}(R)|
        \lesssim_{j,\gamma}
        \frac{L(R)^{\gamma-1}}{R^j},
\]
while
\[
        |a^{(j)}(R)|
        \lesssim_{j,\beta}
        L(R)^{-\beta}R^{-j}.
\]
By Faà di Bruno's formula,
\[
        \left|\frac{\mathrm{d}^j}{\mathrm{d}R^j}
        e^{i\Phi(R)}\right|
        \lesssim_{j,\gamma}
        \left(\frac{L(R)^{\gamma-1}}{R}\right)^j .
\]
Indeed, each term is a product of derivatives of $\Phi$, and the worst
term is the $j$-fold product of $\Phi'(R)$.

Therefore, by Leibniz' rule,
\[
        \left|
        \frac{\mathrm{d}^j}{\mathrm{d}R^j}
        \left(a(R)e^{i\Phi(R)}\right)
        \right|
        \lesssim_{j,\gamma,\beta}
        L(R)^{-\beta}
        \left(\frac{L(R)^{\gamma-1}}{R}\right)^j .
\]
The standard formulae for derivatives of radial functions then give, for
$|\xi|\gg1$,
\[
        |D^\nu m_{\gamma,\beta}(\xi)|
        \lesssim_{\nu,\gamma,\beta}
        L(\xi)^{-\beta}
        \left(
        \frac{L(\xi)^{\gamma-1}}{|\xi|}
        \right)^{|\nu|}.
\]
Since $L(\xi)\sim\log|\xi|$ in the high-frequency region, this is
equivalent to the estimate in the statement.

Since
\[
        \rho(R)^{-1}
        =
        \frac{(\log R)^{\gamma-1}}{R},
\]
this is precisely the pointwise logarithmic Miyachi condition. The localized
Sobolev condition follows by applying these derivative estimates on each
log-subdyadic ball $B$, using $r(B)\sim\rho(R_B)$ and the derivative
bounds for the normalized cutoffs adapted to $B$.
\end{proof}

\begin{remark}[Comparison with classical Mikhlin classes]
\label{rem:comparison-symbol-classes}
The logarithmic Miyachi class contains the usual Mikhlin class with
logarithmic decay. Indeed, if
\[
        |D^\nu m(\xi)|
        \lesssim_\nu
        (\log|\xi|)^{-\beta}|\xi|^{-|\nu|}
\]
for $|\xi|\gg1$, then
\[
        |D^\nu m(\xi)|
        \lesssim_\nu
        (\log|\xi|)^{-\beta}
        \rho(|\xi|)^{-|\nu|},
\]
because
\[
        |\xi|^{-1}
        \leq
        \rho(|\xi|)^{-1}
        =
        \frac{(\log|\xi|)^{\gamma-1}}{|\xi|}
\]
for $|\xi|$ large. Hence the logarithmic Mikhlin class with decay
$(\log|\xi|)^{-\beta}$ is contained in the logarithmic Miyachi class.

The inclusion is strict. Indeed, the radial model
\[
        L(R)=\frac12\log(e^2+R^2),
        \qquad
        m_{\gamma,\beta}(R)
        =
        L(R)^{-\beta}e^{iL(R)^\gamma}
\]
satisfies the logarithmic Miyachi condition, but it does not satisfy the
classical Mikhlin bound with logarithmic decay. Differentiating gives
\[
        \frac{\mathrm{d}}{\mathrm{d}R}m_{\gamma,\beta}(R)
        =
        L'(R)e^{iL(R)^\gamma}
        \left[
        -\beta L(R)^{-\beta-1}
        +
        i\gamma L(R)^{-\beta+\gamma-1}
        \right].
\]
The two terms in brackets are real and purely imaginary, respectively.
Hence there is no cancellation between them, and therefore
\[
        \left|
        \frac{\mathrm{d}}{\mathrm{d}R}m_{\gamma,\beta}(R)
        \right|
        \geq
        c_{\gamma,\beta}
        L(R)^{-\beta+\gamma-1}\frac1R
\]
for all sufficiently large $R$. Since $\gamma>1$, this cannot be bounded
by
\[
        C(\log R)^{-\beta}\frac1R .
\]
Thus the model belongs to the logarithmic Miyachi class but is not a
classical Mikhlin symbol with logarithmic decay.
\end{remark}

\begin{corollary}[Weighted estimate for the model]
\label{cor:model-weighted}
Let $\gamma>1$, $\beta\geq0$, and let
\[
        L(\xi)=\frac12\log(e^2+|\xi|^2),
        \qquad
        m_{\gamma,\beta}(\xi)
        =
        L(\xi)^{-\beta}e^{iL(\xi)^\gamma}.
\]
Let
\[
        m_{\mathrm{hi}}=\chi_{\mathrm{hi}}m_{\gamma,\beta},
        \qquad
        m_{\mathrm{lo}}=(1-\chi_{\mathrm{hi}})m_{\gamma,\beta}
\]
be the fixed inhomogeneous decomposition from Subsection
\ref{subsec:square-functions}. Then
\[
        \int_{\mathbb{R}^d}
        |T_{m_{\mathrm{hi}}}f(x)|^2w(x)\,\mathrm{d}x
        \lesssim
        \int_{\mathbb{R}^d}
        |f(x)|^2
        M^2\bigl(M_{\log,\gamma,\beta}M^4w\bigr)(x)
        \,\mathrm{d}x
\]
for every nonnegative locally integrable weight $w$.
\end{corollary}

\begin{proof}
By Proposition \ref{prop:model-log-miyachi} and Leibniz' rule with the
fixed cutoff $\chi_{\mathrm{hi}}$, the high-frequency symbol
$m_{\mathrm{hi}}$ satisfies Hypothesis
\ref{hyp:localized-log-miyachi}. Indeed, the derivatives falling on
$\chi_{\mathrm{hi}}$ are supported in a fixed annulus and are therefore
absorbed into the structural constant. The result follows directly from
Theorem \ref{thm:weighted-log-multiplier}.
\end{proof}

\begin{corollary}[Sharp $L^p$ bounds for the model]
\label{cor:model-lp}
Let $\gamma>1$, $1<p<\infty$, and $\beta\in\mathbb{R}$. The multiplier
with symbol $m_{\gamma,\beta}(\xi)=L(\xi)^{-\beta}e^{iL(\xi)^\gamma}$
is bounded on $L^p(\mathbb{R}^d)$ if and only if
\[
        \beta\geq d(\gamma-1)\left|\frac12-\frac1p\right|.
\]
\end{corollary}

\begin{proof}
For sufficiency, use the fixed inhomogeneous decomposition
$m_{\mathrm{hi}}=\chi_{\mathrm{hi}}m_{\gamma,\beta}$ and
$m_{\mathrm{lo}}=(1-\chi_{\mathrm{hi}})m_{\gamma,\beta}$.
Proposition \ref{prop:model-log-miyachi} and the derivative bounds for the
fixed cutoff give the localized condition for $m_{\mathrm{hi}}$; the
low-frequency symbol is smooth and compactly supported. Apply Corollary
\ref{cor:lp-log-multiplier-full-range}.

For necessity, first let $1<p<2$. Fix a nonzero smooth cutoff $\chi$
supported in a sufficiently small ball centered at a unit vector $\xi_0$.
For $R\to\infty$, write $u=\log R$, $\Lambda=u^{\gamma-1}$, and define
\[
        \Psi_R(\xi)=\frac{L(R\xi)^\gamma-u^\gamma}{\Lambda},
        \qquad
        A_R(\xi)=u^\beta\chi(\xi)L(R\xi)^{-\beta}.
\]
On every compact set away from the origin,
$L(R\xi)=u+\log|\xi|+O(R^{-2})$ in $C^N$ for each fixed $N$.
Taylor expansion of $(1+v/u)^\gamma$ for bounded $v$ therefore gives
$\Psi_R\to\gamma\log|\xi|$ in $C^N$. Similarly $A_R\to\chi$ in
$C^N$ on the cutoff support. The Hessian of $\gamma\log|\xi|$ has
radial eigenvalue $-\gamma/|\xi|^2$ and tangential eigenvalues
$\gamma/|\xi|^2$. Let $H_0$ be this Hessian at $\xi_0$. By choosing
the support ball small, and then $R$ sufficiently large, the Hessians
$D^2\Psi_R$ are uniformly close to $H_0$ throughout a slightly larger
convex ball $U$ containing the support of $\chi$. In particular,
\[
        |\nabla\Psi_R(\xi+h)-\nabla\Psi_R(\xi)|\gtrsim |h|
\]
whenever $\xi,\xi+h\in U$: integrate the Hessian along the segment and
use the closeness to the fixed invertible matrix $H_0$.

We record the uniform oscillatory estimate needed for the lower bound.
Set
\[
        I_R(x)=\int_{\mathbb{R}^d}A_R(\xi)
        e^{i\Lambda\Psi_R(\xi)+2\pi i x\cdot\xi}\,\mathrm{d}\xi.
\]
After squaring the modulus and changing variables, this becomes
\[
        |I_R(x)|^2=\int_{\mathbb{R}^d}e^{2\pi i x\cdot h}
        \int_{\mathbb{R}^d}A_R(\xi+h)\overline{A_R(\xi)}
        e^{i\Lambda(\Psi_R(\xi+h)-\Psi_R(\xi))}
        \,\mathrm{d}\xi\,\mathrm{d}h.
\]
For $h\ne0$ the phase of the inner integral divided by $|h|$ has
gradient bounded away from zero on its support, by the preceding
inequality. Its derivatives of every fixed positive order are uniformly
bounded there, by the integral formula for differences and the uniform
$C^N$ bounds on $\Psi_R$. The product amplitude has uniformly bounded
derivatives and bounded support. Repeated integration by parts using the
vector field $\nabla\theta/|\nabla\theta|^2$, where
$\theta=(\Psi_R(\cdot+h)-\Psi_R)/|h|$, bounds the inner integral by
$C_N(1+\Lambda|h|)^{-N}$ for every fixed integer $N$; for
$\Lambda|h|\leq1$ use the absolute integral. The integration by parts is
performed on the open set where both points belong to $U$; the amplitude
is compactly supported there. Taking $N>d$ and integrating in $h$ gives
$|I_R(x)|^2\lesssim\Lambda^{-d}$ uniformly in $x$ and $R$.

Consequently, for $K_R=(\chi m_{\gamma,\beta}(R\,\cdot))^\vee$,
\[
        \|K_R\|_\infty\lesssim u^{-\beta}\Lambda^{-d/2},
        \qquad
        \|K_R\|_2\sim u^{-\beta},
\]
where the second estimate follows from Plancherel and $A_R\to\chi\ne0$.
Since $p<2$, the inequality
$\|K_R\|_2^2\leq\|K_R\|_\infty^{2-p}\|K_R\|_p^p$ gives
\[
        \|K_R\|_p\gtrsim u^{-\beta}\Lambda^{d(1/p-1/2)}.
\]
If $T_{m_{\gamma,\beta}}$ is bounded on $L^p$, the invariance of the
multiplier norm under dilation and the fixed test function $\chi^\vee$
imply $\|K_R\|_p\lesssim\|\chi^\vee\|_p$ uniformly in $R$. Since
$\Lambda=u^{\gamma-1}$ and $u\to\infty$, this forces the asserted
necessary inequality for $1<p<2$. For $p>2$, apply the same argument
to the adjoint symbol $\overline{m_{\gamma,\beta}}$ on $L^{p'}$; the
sign change of the phase does not affect the oscillatory estimate.
For $p=2$, Plancherel gives necessity and sufficiency of $\beta\geq0$.
\end{proof}

\begin{remark}[Interpretation and optimality of the threshold]
\label{rem:model-threshold}
The sufficient condition in Corollary \ref{cor:lp-log-multiplier-full-range}
is consistent with the geometric loss
$(a_\gamma(t)/t)^{d|1/2-1/p|}$, which is compensated by the logarithmic
symbol decay. The necessity argument in Corollary \ref{cor:model-lp}
shows that the decay threshold cannot be lowered for the entire class
of localized logarithmic Miyachi symbols. This does not assert necessity
for each individual symbol satisfying the localized condition.
\end{remark}

%% file: 70_lp_bounds_BB.tex
\section{Lebesgue space consequences}
\label{sec:lp-bounds}

We now derive the unweighted consequences of Theorem
\ref{thm:weighted-log-multiplier}. A covering argument gives the strong
bound for the logarithmic geometric maximal operator at the critical
normalization, and hence the strong multiplier bound at the critical decay.

\begin{lemma}[Strong bounds for $M_{\log,\gamma,\beta}$]
\label{lem:log-maximal-Lr}
Let $\gamma>1$ and $1<r<\infty$. If
\[
        2\beta\geq\frac{d(\gamma-1)}{r},
\]
then, for every nonnegative $w\in L^r(\mathbb{R}^d)$,
\[
        \|M_{\log,\gamma,\beta}w\|_{L^r}
        \lesssim \|w\|_{L^r}.
\]
For $r=\infty$, the strong bound holds whenever $\beta\geq0$.
The implicit constant depends on $d,\gamma,\beta,r,t_0$ and the fixed
averaging kernel $\eta$.
\end{lemma}

\begin{proof}
Since $\log(1/t)>1$ for $0<t<t_0$, it suffices to establish the
finite-$r$ bound when $2\beta=d(\gamma-1)/r$. Set
$b(t)=(\log(1/t))^{\gamma-1}$, and choose $0<\tau<t_0$ such that
$\log(1/\tau)>\max\{1,\gamma-1\}$. Extend the aperture factor to all
positive scales by
\[
        b_*(t)=
        \begin{cases}
        b(t),&0<t\leq\tau,\\
        b(\tau),&t>\tau,
        \end{cases}
        \qquad a_*(t)=tb_*(t).
\]
Then $b_*\geq1$ is nonincreasing, $a_*$ is nondecreasing, and
$b_*(t)\lesssim b_*(2t)$ uniformly for $t>0$. The monotonicity of
$a_*$ follows by differentiating $tb(t)$ on $(0,\tau)$; the
comparability follows from the logarithmic formula at small scales and
from the constant extension at large scales. Also
$a_*(2t)\leq2a_*(t)$.

Consider the auxiliary maximal operator
\[
        H_r w(x)=
        \sup_{\substack{t>0\\ |y-x|\leq a_*(t)}}
        b_*(t)^{-d/r}\frac{1}{|B(y,t)|}
        \int_{B(y,t)}w(z)\,\mathrm{d}z.
\]
For locally integrable $w$, the ball averages are continuous in their
center and positive radius. Parametrizing $y=x+a_*(t)v$, $|v|\leq1$,
shows that a countable dense set of parameters suffices, so $H_rw$ is
measurable.

For each integer $k$, consider all balls $B(y,t)$ whose average of $w$
exceeds $2^k$. Their radii are bounded by H\"older's inequality and
$w\in L^r$. The standard disjoint-ball selection by decreasing dyadic
classes of radii gives a countable disjoint subfamily
$B_{k,j}=B(y_{k,j},t_{k,j})$ such that each ball in the family intersects
some $B_{k,j}$ with $t\leq2t_{k,j}$. Indeed, in each class choose a
maximal disjoint subfamily among balls not intersecting any previously
selected ball. Any ball not chosen intersects one from its own or a
larger radius class.

If $|y-x|\leq a_*(t)$ and $B(y,t)$ intersects such a selected ball, then
\[
        |x-y_{k,j}|
        \leq a_*(t)+t+t_{k,j}
        \leq5a_*(t_{k,j}).
\]
Moreover, $t\leq2t_{k,j}$ and the comparability of $b_*$ give
$b_*(t)^{-d}\lesssim b_*(t_{k,j})^{-d}$. For every positive ball average,
choose $k$ so that the average lies in $(2^k,2^{k+1}]$. Taking the
supremum over admissible balls yields
\[
        (H_rw(x))^r
        \lesssim
        \sum_{k\in\mathbb{Z}}2^{kr}\sum_j
        b_*(t_{k,j})^{-d}
        \mathbf{1}_{B(y_{k,j},5a_*(t_{k,j}))}(x).
\]
The normalization cancels the volume of the enlarged ball, since
$b_*(t)^{-d}|B(y,5a_*(t))|=5^d|B(y,t)|$. Consequently, Tonelli's
theorem gives
\[
        \|H_rw\|_{L^r}^r
        \lesssim \sum_{k\in\mathbb{Z}}2^{kr}\sum_j|B_{k,j}|.
\]
For each $k$, the selected balls are disjoint and are contained in
$\{Mw>2^{k-d}\}$: for $z\in B(y,t)$, the inclusion
$B(y,t)\subset B(z,2t)$ gives the required lower bound for the centered
maximal function. Thus
\[
\begin{aligned}
        \|H_rw\|_{L^r}^r
        &\lesssim \sum_{k\in\mathbb{Z}}2^{kr}|\{Mw>2^{k-d}\}|\\
        &\lesssim \|Mw\|_{L^r}^r
        \lesssim \|w\|_{L^r}^r.
\end{aligned}
\]
The second inequality follows by summing the geometric series of levels
pointwise and then integrating; the last is the strong Hardy--Littlewood
maximal inequality for $r>1$ (see \cite{Stein1970SingularIntegrals}).

To return to the smooth averages $A_tw=\eta_t*w$, use the Schwartz
bounds on $\eta$ and decompose into dyadic annuli. For any $N>d$,
\[
        A_tw(y)\lesssim_N\sum_{j\geq0}2^{-j(N-d)}
        \frac{1}{|B(y,2^{j+1}t)|}
        \int_{B(y,2^{j+1}t)}w(z)\,\mathrm{d}z.
\]
Suppose first that $0<t\leq\tau$ and $|x-y|\leq a_*(t)$. For
$u=2^{j+1}t$, monotonicity gives $a_*(u)\geq a_*(t)$ and
$b_*(u)\leq b_*(t)$. Multiplying each average by $b_*(t)^{-d/r}$
therefore bounds it by $H_rw(x)$. Summing the convergent series proves
$b(t)^{-d/r}A_tw(y)\lesssim H_rw(x)$. For $\tau<t<t_0$, the ratio
$a_\gamma(t)/t$ is uniformly bounded; recentering the balls in the
annular estimate at $x$ yields $A_tw(y)\lesssim Mw(x)$ whenever
$|x-y|\leq a_\gamma(t)$. The critical normalization is at most one.
Taking the supremum proves the desired bound by
$M_{\log,\gamma,\beta}w\lesssim H_rw+Mw$ at equality.

Finally, if $r=\infty$ and $\beta\geq0$, then $A_tw(y)\leq\|w\|_{L^\infty}$
and $(\log(1/t))^{-2\beta}\leq1$, which proves the remaining assertion.
\end{proof}

\begin{proposition}[Necessary condition for weak type of the logarithmic maximal operator]
\label{prop:log-maximal-necessity}
Let $\gamma>1$ and $1<r<\infty$. If there is a constant $C_r$ such that
\[
        \|M_{\log,\gamma,\beta}w\|_{L^{r,\infty}(\mathbb{R}^d)}
        \leq
        C_r\|w\|_{L^r(\mathbb{R}^d)}
\]
for every nonnegative $w\in L^r(\mathbb{R}^d)$, then
\[
        2\beta\geq\frac{d(\gamma-1)}{r}.
\]
For $r=\infty$, boundedness on nonnegative $L^\infty$ functions implies
$\beta\geq0$.
\end{proposition}

\begin{proof}
Fix $0<t<t_0$, and set
\[
        w_t=\mathbf{1}_{B(0,t)}.
\]
Since $\eta\geq c\,\mathbf{1}_{B(0,1)}$, a change of variables gives
\[
        A_tw_t(0)
        =
        \int_{|z|\leq t}\eta_t(-z)\,\mathrm{d}z
        =
        \int_{|u|\leq1}\eta(-u)\,\mathrm{d}u
        \geq c_0>0,
\]
where $c_0$ is independent of $t$. If $|x|\leq a_\gamma(t)$, the center
$y=0$ is admissible in Definition \ref{def:log-maximal-candidate}; hence
\[
        M_{\log,\gamma,\beta}w_t(x)
        \geq
        c_0(\log(1/t))^{-2\beta}.
\]
It follows that
\[
        \|M_{\log,\gamma,\beta}w_t\|_{L^{r,\infty}}
        \gtrsim
        (\log(1/t))^{-2\beta}a_\gamma(t)^{d/r}.
\]
Since
\[
        \|w_t\|_{L^r}\sim t^{d/r},
\]
we obtain
\[
        \frac{\|M_{\log,\gamma,\beta}w_t\|_{L^{r,\infty}}}
        {\|w_t\|_{L^r}}
        \gtrsim
        (\log(1/t))^{-2\beta}
        \left(\frac{a_\gamma(t)}{t}\right)^{d/r}
        =
        (\log(1/t))^{-2\beta+d(\gamma-1)/r}.
\]
Uniform weak-type boundedness as $t\downarrow0$ forces
$-2\beta+d(\gamma-1)/r\leq0$, which is the required condition.

For $r=\infty$, apply the operator to $w\equiv1$. Since $A_t1=1$,
boundedness requires
\[
        \sup_{0<t<t_0}(\log(1/t))^{-2\beta}<\infty,
\]
which is equivalent to $\beta\geq0$.
\end{proof}

\begin{remark}[Sharp maximal threshold]
\label{rem:lp-sharpness}
Lemma \ref{lem:log-maximal-Lr} and Proposition
\ref{prop:log-maximal-necessity} show that, for $1<r<\infty$,
$M_{\log,\gamma,\beta}$ is bounded on $L^r$ if and only if
\[
        2\beta\geq\frac{d(\gamma-1)}{r}.
\]
Below this threshold even the weak-type bound fails. The strong critical
bound is the logarithmic counterpart of the same-exponent bounds for the
power geometry in \cite[Theorem 13]{BeltranBennett2017Subdyadic}. Here it
follows from the monotonicity of the aperture and the cancellation of
volumes in the covering argument. This argument is related to the
Carleson measure estimates used for maximal operators associated with
approach regions; see \cite[Proposition 2.7]{GarciaSoria1999ApproachRegions}.
The maximal necessity test alone does not imply necessity for Fourier
multipliers.
\end{remark}

\begin{theorem}[$L^p$ boundedness above $L^2$]
\label{thm:lp-log-multiplier-pgeq2}
Let $\gamma>1$, let $2\leq p<\infty$, and let $\beta\geq0$. If
$p>2$, assume in addition that
\[
        \beta\geq
        d(\gamma-1)
        \left(\frac12-\frac1p\right).
\]
Let
\[
        m=m_{\mathrm{lo}}+m_{\mathrm{hi}},
\]
where $m_{\mathrm{hi}}$ satisfies the localized logarithmic Miyachi
condition of Hypothesis \ref{hyp:localized-log-miyachi}, is supported in
$\{|\xi|\geq2t_0^{-1}\}$, and where $m_{\mathrm{lo}}$ is a smooth
compactly supported multiplier. Then
\[
        \|T_m f\|_{L^p(\mathbb{R}^d)}
        \lesssim
        \|f\|_{L^p(\mathbb{R}^d)} .
\]
The implicit constant depends on $d$, $\gamma$, $\beta$, $p$, $\sigma$,
$t_0$, the localized Miyachi constant $C_m$, on the fixed functions and
cutoffs $\phi$, $\kappa$, and $\eta$, and on the $L^1$ norm of the inverse
Fourier transform of $m_{\mathrm{lo}}$.
\end{theorem}

\begin{proof}
The low-frequency term $T_{m_{\mathrm{lo}}}$ is bounded on $L^p$ by
Young's inequality, since the inverse Fourier transform of
$m_{\mathrm{lo}}$ is Schwartz. It remains to treat $T_{m_{\mathrm{hi}}}$.

The case $p=2$ follows from Theorem
\ref{thm:weighted-log-multiplier} with $w\equiv1$.

Assume now that $p>2$, and set
\[
        r=\left(\frac p2\right)'=\frac{p}{p-2}.
\]
By duality,
\[
        \|T_{m_{\mathrm{hi}}} f\|_{L^p}^2
        =
        \bigl\||T_{m_{\mathrm{hi}}} f|^2\bigr\|_{L^{p/2}}
        =
        \sup_{\|w\|_{L^r}=1}
        \int_{\mathbb{R}^d}|T_{m_{\mathrm{hi}}} f(x)|^2w(x)\,\mathrm{d}x,
\]
where the supremum is taken over nonnegative $w\in L^r$.

Applying Theorem \ref{thm:weighted-log-multiplier}, we obtain
\[
        \int_{\mathbb{R}^d}|T_{m_{\mathrm{hi}}} f(x)|^2w(x)\,\mathrm{d}x
        \lesssim
        \int_{\mathbb{R}^d}|f(x)|^2
        M^2\bigl(M_{\log,\gamma,\beta}M^4w\bigr)(x)\,\mathrm{d}x .
\]
Since $M$ is bounded on $L^r$ and, by Lemma
\ref{lem:log-maximal-Lr}, $M_{\log,\gamma,\beta}$ is bounded on $L^r$
whenever
\[
        2\beta\geq\frac{d(\gamma-1)}{r},
\]
we have
\[
        \bigl\|M^2(M_{\log,\gamma,\beta}M^4w)\bigr\|_{L^r}
        \lesssim
        \|w\|_{L^r}.
\]
H\"older's inequality therefore gives
\[
        \int_{\mathbb{R}^d}|T_{m_{\mathrm{hi}}} f(x)|^2w(x)\,\mathrm{d}x
        \lesssim
        \|\,|f|^2\,\|_{L^{p/2}}\|w\|_{L^r}
        =
        \|f\|_{L^p}^2\|w\|_{L^r}.
\]
Taking the supremum over $\|w\|_{L^r}=1$ yields
\[
        \|T_{m_{\mathrm{hi}}} f\|_{L^p}
        \lesssim
        \|f\|_{L^p}.
\]
Combining this estimate with the $L^p$-boundedness of
$T_{m_{\mathrm{lo}}}$ gives the asserted bound for $T_m$.

Finally,
\[
        2\beta\geq\frac{d(\gamma-1)}{r}
\]
is equivalent to
\[
        \beta\geq
        \frac{d(\gamma-1)}{2r}
        =
        d(\gamma-1)
        \left(\frac12-\frac1p\right),
\]
which is the stated condition for $p>2$. The estimates are first obtained
on Schwartz functions, using the extended $L^{p/2}$ norm if necessary;
density then gives the bounded extension to $L^p$.
\end{proof}

\begin{corollary}[$L^p$ boundedness by duality]
\label{cor:lp-log-multiplier-full-range}
Let $\gamma>1$, $1<p<\infty$, and let $\beta\geq0$. Let
\[
        m=m_{\mathrm{lo}}+m_{\mathrm{hi}},
\]
where $m_{\mathrm{lo}}$ is a smooth compactly supported multiplier and
$m_{\mathrm{hi}}$ is supported in $\{|\xi|\geq2t_0^{-1}\}$ and satisfies
Hypothesis \ref{hyp:localized-log-miyachi}. If $p\neq2$, assume in addition
that
\[
        \beta\geq
        d(\gamma-1)
        \left|\frac12-\frac1p\right|.
\]
Then
\[
        \|T_m f\|_{L^p(\mathbb{R}^d)}
        \lesssim
        \|f\|_{L^p(\mathbb{R}^d)} .
\]
\end{corollary}

\begin{proof}
For $2\leq p<\infty$, this is Theorem
\ref{thm:lp-log-multiplier-pgeq2}.

Assume $1<p<2$. Then $p'>2$. The adjoint multiplier decomposes as
\[
        \overline m
        =
        \overline{m_{\mathrm{lo}}}
        +
        \overline{m_{\mathrm{hi}}},
\]
where $\overline{m_{\mathrm{lo}}}$ is again smooth and compactly
supported, while $\overline{m_{\mathrm{hi}}}$ satisfies the localized
logarithmic Miyachi condition by Lemma
\ref{lem:log-miyachi-conjugation}. Applying Theorem
\ref{thm:lp-log-multiplier-pgeq2} to $T_{\overline m}=T_m^*$ on
$L^{p'}$, and then dualizing, gives the $L^p$ bound for $T_m$.
Since
\[
        \frac12-\frac1{p'}
        =
        \frac1p-\frac12,
\]
the stated condition is exactly the condition needed for the adjoint bound.
\end{proof}

\begin{remark}[Comparison with Besov multiplier and sparse criteria]
\label{rem:seeger-comparison}
The unweighted bound in Corollary \ref{cor:lp-log-multiplier-full-range}
can also be deduced from the classical Besov multiplier criterion of
Seeger \cite[Corollary 1.2]{Seeger1990SingularConvolution}. We give the
reduction to make explicit the relation between the two hypotheses.

Let $1<p<2$ and $s=d(1/p-1/2)$. Choose a finite exponent $q$ such that
\[
        \frac{d}{s}<q<\frac{2\sigma}{s}.
\]
This is possible since $\sigma>d/2$; also $q>2$. Let $\varphi$ be a
fixed smooth cutoff supported in $\{1/4<|\xi|<4\}$ and positive on
$\{1/2<|\xi|<2\}$. Seeger's criterion requires the uniform bound of
$\varphi m(R\,\cdot)$ in the inhomogeneous Besov space $B^s_{q,1}$,
for all $R>0$, together with boundedness of the symbol.

For $R$ sufficiently large put $\delta_R=(\log R)^{1-\gamma}$ and
$A_R=(\log R)^{-\beta}$, and write $F_R=\varphi m(R\,\cdot)$. On the
support of $\varphi$ the low-frequency symbol is absent. A lattice
partition of this fixed annulus at radius comparable to $\delta_R$
decomposes $F_R=\sum_i F_{R,i}$ into $O(\delta_R^{-d})$ pieces with
bounded overlap. In the original frequency variable the corresponding
bumps are adapted to log-subdyadic balls: their radii are comparable to
$R\delta_R$, and their centers have modulus comparable to $R$. The
lattice spacing and bump supports can be taken sufficiently small fixed
multiples of $R\delta_R$ so that the bumps are supported inside the
balls in Definition \ref{def:log-subdyadic-ball}. Multiplication by
$\varphi(\cdot/R)$ preserves the normalized derivative bounds since
$R\delta_R\leq R$. Hypothesis \ref{hyp:localized-log-miyachi} and
rescaling give
\[
        \|F_{R,i}\|_2\lesssim A_R\delta_R^{d/2},
        \qquad
        \|F_{R,i}\|_{\dot H^\sigma}
        \lesssim A_R\delta_R^{d/2-\sigma}.
\]
After dilation to unit scale, the Sobolev embedding for $\sigma>d/2$
and bounded overlap give $\|F_R\|_\infty\lesssim A_R$. The embedding
here also follows directly by applying Cauchy--Schwarz to Fourier
inversion with the integrable weight $(1+|\xi|^2)^{-\sigma}$.

For an integer $J>\sigma$, let $\Delta_h^J$ denote the difference of
order $J$ with step $h$. The supports of the differences of the pieces
have overlap bounded by $J+1$ times the original overlap bound, since
only $J+1$ translates occur. Plancherel applied to each piece yields
\[
        \|\Delta_h^J F_R\|_2
        \lesssim A_R\min\{1,(|h|/\delta_R)^\sigma\}.
\]
Indeed, the difference multiplier is bounded by a constant times
$\min\{1,|h|^\sigma|\xi|^\sigma\}$, and the squares of the localized
bounds sum to the stated estimate. Interpolation with the uniform
$L^\infty$ bound gives
\[
        \|\Delta_h^J F_R\|_q
        \lesssim A_R\min\{1,(|h|/\delta_R)^{2\sigma/q}\}.
\]
In the standard characterization of $B^s_{q,1}$ by differences
(see \cite{Triebel1983Theory}), splitting the integral at $\delta_R$
then yields, since $0<s<2\sigma/q$,
\[
        \|F_R\|_{B^s_{q,1}}
        \lesssim A_R\delta_R^{-s}
        = (\log R)^{-\beta+s(\gamma-1)}.
\]
For bounded $R$ away from zero, the high-frequency portion on the
annulus has a finite covering by adapted balls, which gives a uniform
bound by the same argument; for sufficiently small $R$ it is absent.
The smooth compactly supported low-frequency portion has uniformly
bounded derivatives on the rescaled annuli on which it is nonzero.
The local Sobolev embedding also gives boundedness of $m$ for $\beta\geq0$.
The criterion of Seeger therefore applies when $\beta\geq s(\gamma-1)$,
and duality gives the corresponding range above $L^2$.

The same difference estimates also relate the localized hypothesis to
recent sparse multiplier criteria. Let $1<p<2$, $s=d(1/p-1/2)$ as above,
and assume the strict inequality $\beta>s(\gamma-1)$. Put $u=d/s$ and
choose $\alpha$ such that
\[
        s<\alpha<\min\left\{\frac{2\sigma}{u},
        \frac{\beta}{\gamma-1}\right\}.
\]
Such a choice is possible since $\sigma>d/2$. Applying the difference
argument with exponents $u$ and $\alpha$ in place of $q$ and $s$ gives
\[
        \sup_{R>0}\|\varphi m(R\,\cdot)\|_{B^\alpha_{u,1}}<\infty.
\]
The embedding of $B^\alpha_{u,1}$ into the Bessel potential space
$H^\alpha_u$ (see \cite{Triebel1983Theory}) then verifies the hypothesis of
\cite[Proposition 7.5]{BeltranRoosSeeger2024MultiScale} with
$1/u=1/p-1/2$ and output exponent $2$. Thus the same localized hypothesis
also gives a $(p,2)$ sparse bound in this strictly supercritical range.
The reduction to Seeger's criterion includes the critical decay; the
above reduction to a sparse criterion uses strict inequality. These
comparisons do not identify the resulting bounds with the pointwise
square-function or explicit geometric weighted inequalities proved here.
\end{remark}

\begin{remark}[No power-type Sobolev gain asserted]
\label{rem:no-lp-lq-gain}
The aperture ratio
\[
        a_\gamma(t)/t=(\log(1/t))^{\gamma-1}
\]
produces a logarithmic geometric loss rather than a power of $t^{-1}$. The
consequences proved here are same-exponent Lebesgue bounds with a
logarithmic decay threshold; no $L^p\to L^q$ gain with $q>p$ is asserted.
\end{remark}

%% file: BB_log.bib
@article{Hirschman1959MultiplierTransformations,
  author  = {Hirschman, Jr., I. I.},
  title   = {{On multiplier transformations}},
  journal = {Duke Mathematical Journal},
  volume  = {26},
  number  = {2},
  pages   = {221--242},
  year    = {1959},
  doi     = {10.1215/S0012-7094-59-02623-7}
}

@book{Wainger1965SpecialTrigonometricSeries,
  author    = {Wainger, S.},
  title     = {{Special Trigonometric Series in \(k\)-Dimensions}},
  series    = {Memoirs of the American Mathematical Society},
  number    = {59},
  publisher = {American Mathematical Society},
  address   = {Providence, RI},
  year      = {1965}
}

@article{Hormander1960TranslationInvariant,
  author  = {H{\"o}rmander, L.},
  title   = {{Estimates for translation invariant operators in \(L^p\) spaces}},
  journal = {Acta Mathematica},
  volume  = {104},
  pages   = {93--140},
  year    = {1960},
  doi     = {10.1007/BF02547187}
}

@article{Fefferman1970StronglySingular,
  author  = {Fefferman, C.},
  title   = {{Inequalities for strongly singular convolution operators}},
  journal = {Acta Mathematica},
  volume  = {124},
  pages   = {9--36},
  year    = {1970},
  doi     = {10.1007/BF02394567}
}

@article{FeffermanStein1972HpSpaces,
  author  = {Fefferman, C. and Stein, E. M.},
  title   = {{\({H}^p\) spaces of several variables}},
  journal = {Acta Mathematica},
  volume  = {129},
  pages   = {137--193},
  year    = {1972},
  doi     = {10.1007/BF02392215}
}

@article{Sjolin1981OscillatingKernels,
  author  = {Sj{\"o}lin, P.},
  title   = {{Convolution with oscillating kernels}},
  journal = {Indiana University Mathematics Journal},
  volume  = {30},
  number  = {1},
  pages   = {47--55},
  year    = {1981},
  doi     = {10.1512/iumj.1981.30.30004}
}

@article{NagelStein1984MaximalApproachRegions,
  author  = {Nagel, A. and Stein, E. M.},
  title   = {{On certain maximal functions and approach regions}},
  journal = {Advances in Mathematics},
  volume  = {54},
  number  = {1},
  pages   = {83--106},
  year    = {1984},
  doi     = {10.1016/0001-8708(84)90038-0}
}

@article{Bennett2014OptimalControl,
  author  = {Bennett, J.},
  title   = {{Optimal control of singular {F}ourier multipliers by maximal operators}},
  journal = {Analysis \& PDE},
  volume  = {7},
  number  = {6},
  pages   = {1317--1338},
  year    = {2014},
  doi     = {10.2140/apde.2014.7.1317}
}

@article{RubioDeFrancia1985LittlewoodPaley,
  author  = {Rubio de Francia, J. L.},
  title   = {{A {L}ittlewood--{P}aley inequality for arbitrary intervals}},
  journal = {Revista Matem{\'a}tica Iberoamericana},
  volume  = {1},
  number  = {2},
  pages   = {1--14},
  year    = {1985}
}

@article{Wilson1989WeightedSquareFunction,
  author  = {Wilson, J. M.},
  title   = {{Weighted norm inequalities for the continuous square function}},
  journal = {Transactions of the American Mathematical Society},
  volume  = {314},
  number  = {2},
  pages   = {661--692},
  year    = {1989}
}

@incollection{LeeRogersSeeger2014SquareFunctions,
  author    = {Lee, S. and Rogers, K. M. and Seeger, A.},
  title     = {{Square functions and maximal operators associated with radial {F}ourier multipliers}},
  booktitle = {Advances in Analysis: The Legacy of Elias M. Stein},
  series    = {Princeton Mathematical Series},
  volume    = {50},
  pages     = {273--302},
  publisher = {Princeton University Press},
  address   = {Princeton, NJ},
  year      = {2014}
}

@article{BeltranBennett2017Subdyadic,
  author  = {Beltran, D. and Bennett, J.},
  title   = {{Subdyadic square functions and applications to weighted harmonic analysis}},
  journal = {Advances in Mathematics},
  volume  = {307},
  pages   = {72--99},
  year    = {2017},
  doi     = {10.1016/j.aim.2016.11.018}
}

@book{Stein1970SingularIntegrals,
  author    = {Stein, E. M.},
  title     = {{Singular Integrals and Differentiability Properties of Functions}},
  publisher = {Princeton University Press},
  year      = {1970}
}

@book{Stein1970Topics,
  author    = {Stein, E. M.},
  title     = {{Topics in Harmonic Analysis Related to the Littlewood--Paley Theory}},
  series    = {Annals of Mathematics Studies},
  volume    = {63},
  publisher = {Princeton University Press},
  address   = {Princeton, NJ},
  year      = {1970}
}

@article{Miyachi1981Singular,
  author  = {Miyachi, A.},
  title   = {{On some singular {F}ourier multipliers}},
  journal = {Journal of the Faculty of Science, University of Tokyo. Section IA. Mathematics},
  volume  = {28},
  number  = {2},
  pages   = {267--315},
  year    = {1981}
}

@article{Miyachi1980FourierMultipliersHp,
  author  = {Miyachi, A.},
  title   = {{On some {F}ourier multipliers for \({H}^p(\mathbb{R}^n)\)}},
  journal = {Journal of the Faculty of Science, University of Tokyo. Section IA. Mathematics},
  volume  = {27},
  number  = {1},
  pages   = {157--179},
  year    = {1980}
}

@article{Seeger1990SingularConvolution,
  author  = {Seeger, A.},
  title   = {{Remarks on singular convolution operators}},
  journal = {Studia Mathematica},
  volume  = {97},
  number  = {2},
  pages   = {91--114},
  year    = {1990},
  doi     = {10.4064/sm-97-2-91-114}
}

@book{Triebel1983Theory,
  author    = {Triebel, H.},
  title     = {{Theory of Function Spaces}},
  publisher = {Birkh\"auser},
  address   = {Basel},
  year      = {1983}
}

@article{Seeger1988LimitHormander,
  author  = {Seeger, A.},
  title   = {{A limit case of the H{\"o}rmander multiplier theorem}},
  journal = {Monatshefte f{\"u}r Mathematik},
  volume  = {105},
  number  = {2},
  pages   = {151--160},
  year    = {1988},
  doi     = {10.1007/BF01501167}
}

@article{Seeger1989NearL1,
  author  = {Seeger, A.},
  title   = {{Estimates near $L^1$ for Fourier multipliers and maximal functions}},
  journal = {Archiv der Mathematik},
  volume  = {53},
  number  = {2},
  pages   = {188--193},
  year    = {1989},
  doi     = {10.1007/BF01198570}
}

@article{SeegerTrebels2019LorentzSobolev,
  author  = {Seeger, A. and Trebels, W.},
  title   = {{Embeddings for spaces of Lorentz--Sobolev type}},
  journal = {Mathematische Annalen},
  volume  = {373},
  pages   = {1017--1056},
  year    = {2019},
  doi     = {10.1007/s00208-018-1730-8}
}

@article{GarciaSoria1999ApproachRegions,
  author  = {Garc{\'i}a, J. L. and Soria, J.},
  title   = {{Weighted inequalities and the shape of approach regions}},
  journal = {Studia Mathematica},
  volume  = {133},
  number  = {3},
  pages   = {261--274},
  year    = {1999},
  doi     = {10.4064/sm-133-3-261-274}
}

@article{BeltranRoosSeeger2024MultiScale,
  author  = {Beltran, D. and Roos, J. and Seeger, A.},
  title   = {{Multi-scale Sparse Domination}},
  journal = {Memoirs of the American Mathematical Society},
  volume  = {298},
  number  = {1491},
  year    = {2024},
  doi     = {10.1090/memo/1491}
}

@article{VergaraSubdyadicGabor,
  author  = {Vergara, V.},
  title   = {{Subdyadic time-frequency analysis: Gabor frames, modulation spaces, and Miyachi multipliers}},
  journal = {Journal of Fourier Analysis and Applications},
  year    = {2026},
  note    = {To appear. arXiv:2511.17740},
  eprint  = {2511.17740},
  archivePrefix = {arXiv},
  primaryClass  = {math.FA}
}
